\documentclass[11pt, onecolumn]{article}

\usepackage{constants}

\usepackage{stmaryrd}
\usepackage{tikz}
\usetikzlibrary{shapes.misc}
\tikzset{cross/.style={cross out, draw=black, minimum size=2*(#1-\pgflinewidth), inner sep=0pt, outer sep=0pt},
cross/.default={1pt}}

\usepackage{pstricks}
\usepackage[utf8]{inputenc}
\usepackage[T1]{fontenc}
\usepackage{lmodern}
\usepackage{xr}
\usepackage{bold-extra}
\usepackage{dsfont}

\usepackage{amsmath}
\usepackage{amssymb}
\usepackage{mathrsfs}
\usepackage{mathtools}

\usepackage{amsfonts}
\usepackage{amsthm}

\usepackage{enumerate}
\usepackage{multirow}
\usepackage{graphicx}
\usepackage{hyperref}
\usepackage{xcolor}
\hypersetup{
    colorlinks,
    linkcolor={black},
    citecolor={black},
    urlcolor={black}
}

\usepackage[top=2.5cm,bottom=2.5cm,left=2cm,right=2cm,marginparwidth=2.5cm]{geometry}
\reversemarginpar

\makeatletter
\renewcommand{\paragraph}{%
  \@startsection{paragraph}{4}%
  {\z@}{1.5ex \@plus 1ex \@minus .2ex}{-1em}%
  {\normalfont\normalsize\bfseries}%
}
\makeatother

{\theoremstyle{plain}
\newtheorem{Proposition}{\textbf{Proposition}}[section]

\newtheorem{Lemma}[Proposition]{\textbf{Lemma}}

\newtheorem{Corollary}[Proposition]{\textbf{Corollary}}
 }

{\theoremstyle{plain}
\newtheorem{Theorem}[Proposition]{Theorem}}

{\theoremstyle{definition}
\newtheorem{Definition}[Proposition]{Definition}}

{\theoremstyle{remark}

\newtheorem{Remark}[Proposition]{Remark}

}

\numberwithin{equation}{section}

\newcommand{\eps}{\varepsilon}

\newcommand{\bbE}{\mathbb{E}}

\newcommand{\bbX}{\mathbb{X}}

\newcommand{\N}{\mathbb{N}}
\newcommand{\Z}{\mathbb{Z}}
\newcommand{\R}{\mathbb{R}}

\newcommand{\cA}{\mathcal{A}}

\newcommand{\cE}{\mathcal{E}}

\newcommand{\cG}{\mathcal{G}}

\newcommand{\rA}{\mathrm{A}}

\newcommand{\rG}{\mathrm{G}}

\newcommand{\rL}{\mathrm{L}}

\newcommand{\rP}{\mathrm{P}}

\let\limsup\relax
\let\liminf\relax
\DeclareMathOperator* \limsup {\overline{lim}}
\DeclareMathOperator* \liminf {\underline{lim}}

\DeclareMathOperator*{\argmin}{arg\,min}

\let\originalleft\left
\let\originalright\right
\renewcommand{\left}{\mathopen{}\mathclose\bgroup\originalleft}
\renewcommand{\right}{\aftergroup\egroup\originalright}

\newcommand{\p}[1]{\left( #1 \right)}
\newcommand{\acc}[1]{\left\{ #1 \right\}}
\newcommand{\cro}[1]{\left[ #1 \right]}

\newcommand{\set}[2]{\acc{#1 \;\middle\vert\; #2 } }

\newcommand{\ind}[1]{\mathds{1}_{#1}}

\newcommand{\dpe}{\coloneqq}

\newcommand{\eol}{\nonumber\\}

\def\restriction#1#2{\mathchoice
              {\setbox1\hbox{${\displaystyle #1}_{\scriptstyle #2}$}
              \restrictionaux{#1}{#2}}
              {\setbox1\hbox{${\textstyle #1}_{\scriptstyle #2}$}
              \restrictionaux{#1}{#2}}
              {\setbox1\hbox{${\scriptstyle #1}_{\scriptscriptstyle #2}$}
              \restrictionaux{#1}{#2}}
              {\setbox1\hbox{${\scriptscriptstyle #1}_{\scriptscriptstyle #2}$}
              \restrictionaux{#1}{#2}}}
\def\restrictionaux#1#2{{#1\,\smash{\vrule height .8\ht1 depth .85\dp1}}_{\,#2}} 

\newcommand{\norme}[2][]{\left\| #2 \right\|_{#1}}

\newcommand{\ceil}[1]{\!\left\lceil #1 \right\rceil\!}
\newcommand{\floor}[1]{\!\left\lfloor #1 \right\rfloor\!}

\newcommand{\ball}[2][]{\mathrm{B}_{#1}\p{#2}}
\newcommand{\clball}[2][]{\overline{\mathrm{B}}_{#1}\p{#2}}

\newcommand{\intervalle}[4]{#1#2\,,#3#4}

\newcommand{\intervalleff}[2]{\intervalle{\left[}{#1}{#2}{\right]}}
\newcommand{\intervalleof}[2]{\intervalle{\left(}{#1}{#2}{\right]}}
\newcommand{\intervallefo}[2]{\intervalle{\left[}{#1}{#2}{\right)}}
\newcommand{\intervalleoo}[2]{\intervalle{\left(}{#1}{#2}{\right)}}

\newcommand{\intint}[2]{\left\llbracket#1\,,#2\right\rrbracket}

\renewcommand{\d}{\mathrm{d}}

\DeclareMathOperator \diam {diam}

\DeclareMathOperator \Pol {Pol}

\newcommand{\base}[1]{\mathrm e_{#1}}

\newcommand\ps[2]{\left\langle #1, #2 \right\rangle}

\newcommand{\Leb}[1][d]{\operatorname{Leb}^{#1} }

\newcommand{\Dirac}[1]{\delta_{#1}}

\newcommand{\E}[2][]{\mathbb{E}_{#1} \left[ #2\right]}

\newcommand{\Pb}[2][]{\mathbb{P}_{#1}\left( #2\right)}

\newcommand{\Path}[1]{\overset{#1}{\rightsquigarrow}}

\newcommand{\concat}{\mathbin{*}}

\newcommand{\pro}{\mathcal{N}}
\newcommand{\projpro}{\mathcal{N^*}}
\newcommand{\proDirac}{ \overline{ \mathcal{N} } }
\newcommand{\progen}{\mathfrak{N}}

\newcommand{\ProSpace}[1][\R^d \times \intervalleoo0\infty]{\mathbf{N}\p{#1}}
\newcommand{\ProSpaceFinite}[1][\R^d \times \intervalleoo0\infty]{\mathbf{N}_{<\infty}\p{#1}}
\newcommand{\DiracSpace}[1][\R^d \times \intervalleoo0\infty]{\mathbf{N}_{\le 1}\p{#1} }

\newcommand{\AUXAnimal}{\mathcal{A}}
\newcommand{\AUXPath}{\mathcal{P}}

\newcommand{\AUXGeneric}{\mathcal{G}}
\newcommand{\AUXMassAnimal}{\mathrm{A}}
\newcommand{\AUXMassPath}{\mathrm{P}}

\newcommand{\AUXMassGeneric}{\mathrm{G}}
\newcommand{\AUXLimAnimal}{\mathbf{A}}

\newcommand{\AUXLimGeneric}{\mathbf{G}}
\newcommand{\AUXBiancre}[3]{\p{#1 \leftrightarrow #2, #3}}

\newcommand{\AUXDiamant}[2]{#1^{#2}}
\newcommand{\AUXFree}[1]{#1}

\newcommand{\Mass}[1]{\mathbf{m}\p{#1}}
\newcommand{\Masspen}[2][q]{\mathbf{m}^{(#1)}\p{#2}  }

\newcommand{\SetADF}[3]{\AUXFree{\AUXAnimal}\AUXBiancre{#1}{#2}{#3} }
\newcommand{\SetAUF}[1]{\AUXFree{\AUXAnimal}\p{#1} }

\newcommand{\SetPDC}[4]{\AUXDiamant{\AUXPath}{#1}\AUXBiancre{#2}{#3}{#4} }
\newcommand{\SetPDF}[3]{\AUXFree{\AUXPath}\AUXBiancre{#1}{#2}{#3} }
\newcommand{\SetPUF}[1]{\AUXFree{\AUXPath}\p{#1} }

\newcommand{\SetGDC}[4]{\AUXDiamant{\AUXGeneric}{#1}\AUXBiancre{#2}{#3}{#4} }
\newcommand{\SetGDF}[3]{\AUXFree{\AUXGeneric}\AUXBiancre{#1}{#2}{#3} }

\newcommand{\SetGUF}[1]{\AUXFree{\AUXGeneric}\p{#1} }

\newcommand{\SetADCalt}[4]{\AUXDiamant{\AUXAnimal}{*, #1}\AUXBiancre{#2}{#3}{#4} }
\newcommand{\SetADFalt}[3]{\AUXFree{\AUXAnimal}^{*}\AUXBiancre{#1}{#2}{#3} }
\newcommand{\SetAUFalt}[1]{\AUXFree{\AUXAnimal}^{*}\p{#1} }

\newcommand{\MassN}[1]{\overline{ \mathrm A }\p{#1} }

\newcommand{\MassADF}[3]{\AUXFree{\AUXMassAnimal}\AUXBiancre{#1}{#2}{#3} }

\newcommand{\MassAUF}[1]{\AUXFree{\AUXMassAnimal}\p{#1} }

\newcommand{\MassPDC}[4]{\AUXDiamant{\AUXMassPath}{#1}\AUXBiancre{#2}{#3}{#4} }
\newcommand{\MassPDF}[3]{\AUXFree{\AUXMassPath}\AUXBiancre{#1}{#2}{#3} }

\newcommand{\MassPUF}[1]{\AUXFree{\AUXMassPath}\p{#1} }

\newcommand{\MassGDC}[4]{\AUXDiamant{\AUXMassGeneric}{#1}\AUXBiancre{#2}{#3}{#4} }
\newcommand{\MassGDF}[3]{\AUXFree{\AUXMassGeneric}\AUXBiancre{#1}{#2}{#3} }

\newcommand{\MassGUF}[1]{\AUXFree{\AUXMassGeneric}\p{#1} }

\newcommand{\MassADCpen}[5]{\AUXFree{\AUXMassAnimal}^{(#5), #1}\AUXBiancre{#2}{#3}{#4} }

\newcommand{\MassADFpen}[4]{\AUXFree{\AUXMassAnimal}^{(#4)}\AUXBiancre{#1}{#2}{#3} }
\newcommand{\MassAUFpen}[2]{\AUXFree{\AUXMassAnimal}^{(#2)}\p{#1} }

\newcommand{\hatMassAUFpen}[2]{\AUXFree{\hat \AUXMassAnimal}^{(#2)}\p{#1} }

\newcommand{\MassLAUF}[1]{\mathrm{A^L}\p{#1} }

\newcommand{\LimMassG}{\AUXLimGeneric }

\newcommand{\LimMassA}{\AUXLimAnimal }

\newcommand{\LimMassLA}{\AUXLimAnimal^{\mathrm L}}

\newcommand{\FdT}{J }
\newcommand{\FdTDiamond}[1][\delta]{J^{ #1}}

\newcommand{\Diamant}[3]{\operatorname{Diam}^{#1}(#2 \leftrightarrow #3)}

\newcommand{\Cone}[3]{\operatorname{Cone}^{#1}(#2, #3)}
\newcommand{\Gap}[2]{\Diamant{1}{#1}{#2} }

\title{Concentration inequalities and large deviations for continuous greedy animals and paths }
\author{Julien \textsc{Verges}\footnote{julien.verges@univ-tours.fr}}
\date{\today}

\begin{document}

\maketitle

\abstract{Consider the continuous greedy paths model: given a $d$-dimensional Poisson point process with positive marks interpreted as masses, let $\mathrm P (\ell)$ denote the maximum mass gathered by a path of length $\ell$ starting from the origin. It is known that $\mathrm P(\ell) / \ell$ converges a.s.\ to a deterministic constant $\mathbf P$. We show that the lower-tail deviation probability for $\mathrm P(\ell)$ has order $\exp\p{-\ell^2}$ and, under exponential moment assumption on the mass distribution, that the upper-tail deviation probability has order $\exp\p{-\ell}$. In the latter regime, we prove the existence and some properties --- notably, convexity --- of the corresponding rate function. An immediate corollary is the large deviation principle at speed $\ell$ for $\mathrm P(\ell)$. Along the proof we show an upper-tail concentration inequality in the case where marks are bounded. All of the above also holds for greedy animals and have versions where the paths or animals involved have two anchors instead of one.}

\section{Introduction}
\label{sec : Intro}

\subsection{Background and aim of the paper}
\label{subsec : Intro/Background}

Let us present the model of \emph{lattice greedy animals }and \emph{paths}, first introduced by Cox, Gandolfi, Griffin and Kesten in 1993-94 \cite{Cox93,Gan94}. Consider an integer $d\ge 2$ and the standard $\Z^d$ lattice, i.e.\ the graph whose vertex set is $\Z^d$, in which two vertices are neighbours if and only if their Euclidean distance is exactly $1$. We denote by $\bbE^d$ its edge set. Let $(\Mass{v})_{v\in \Z^d}$ be an i.i.d.\ family of nonnegative random variables, with common distribution $\nu$. The variable $\Mass v$ is called the \emph{mass} of the vertex $v$. The mass of a \emph{lattice animal} (i.e.\  a connected finite subset of $\Z^d$) is defined as the sum of the masses of its vertices. Under a certain moment condition, Cox, Gandolfi, Griffin and Kesten showed that the maximal mass $\MassLAUF n$ over lattice animals of cardinal $n$, containing the origin satisfies the law of large numbers
\begin{equation}
	\label{eqn : Intro/Background/LLN}
	\frac{\MassLAUF n}{n} \xrightarrow[n\to\infty]{\text{a.s.\ and }\rL^1} \LimMassLA,
\end{equation}
where $\LimMassLA$ is a deterministic constant. Animals realizing the maximum in $\MassLAUF n$ are called \emph{greedy}. Martin \cite{Mar02} later proved the same result, under the weaker assumption
\begin{equation}
	\label{ass : Greedy} \tag{Greedy}
	\int_0^\infty \nu\p{\intervallefo t\infty}^{1/d}\d t< \infty.
\end{equation}
Besides, the result also holds for self-avoiding lattice paths originating from $0$.

A \emph{continuous} analogue of greedy paths was introduced as a tool by Gouéré-Marchand \cite{Gou08} for the study of a growth model and Gouéré-Théret \cite{Gou17} for the study of Boolean percolation. In this context a continuous path is a finite-length polygonal line in $\R^d$, and the family of masses $(\Mass{v})_{v\in \Z^d}$ is replaced by a marked Poisson point process on $\R^d$ (see Definition~\ref{def : Intro/Formal/ContinuousPath}). In \cite{Gou08} it is shown that if the mark distribution satisfies~\eqref{ass : Greedy}, then the maximal mass of a continous path with length at most $\ell$ and originating from the origin has at most linear growth. The author \cite{Article3, Article2} recently proved a functional version of~\eqref{eqn : Intro/Background/LLN} for continuous greedy paths and several variants of continuous lattice animals on $\R^d$, Theorem~\ref{thm : Intro/LLN} is a simplification of which. The formal definition of this model is given in Section~\ref{subsec : Intro/Formal}. Since we will only study the continuous model, except stated otherwise, "animal" and "path" will mean "continuous animal" and "continuous path".

In this paper we study the speed of convergence to $0$ of the probability of the \emph{lower-tail} (resp.\ \emph{upper-tail}) \emph{large deviation events}, i.e.\  the events where the greedy animal/path of length $\ell$'s mass $\MassAUF\ell$ is linearly smaller (resp.\ larger) than the limit constant $\LimMassA$. Theorem~\ref{thm : MAIN/LowerTail} states that for all $\zeta < \LimMassA$, for large $\ell$,
\begin{equation}
	\Pb{\MassAUF\ell \le \ell \zeta} \le \exp\p{-c \ell ^d },	
\end{equation}
where $c>0$ depends on $\zeta$. Theorem~\ref{thm : MAIN/UpperTail} gives a large deviation estimate for the upper-tail deviations of $\MassAUF\ell$ under an exponential moment condition: essentially, for all $\zeta > \LimMassA$, for large $\ell$,
\begin{equation}
	\label{eqn : Intro/Background/LDP_layman}
	\Pb{\MassAUF\ell \ge \ell \zeta} \approx \exp\p{-c' \ell },	
\end{equation}
where $c'$ depends on $\zeta$. Finally, Corollary~\ref{cor : LDP} gives a large deviation principle at speed $\ell$ that extends \eqref{eqn : Intro/Background/LDP_layman}. These results also hold for animals (resp.\ paths) constrained to contain (resp.\ to end at) the point $\ell \beta \base 1$, for some $\beta \in \intervallefo01$.

\subsection{Formal definition of the model}
\label{subsec : Intro/Formal}

Let $d\ge 2$ be an integer. We endow $\R^d$ with the Euclidean norm $\norme\cdot$ and denote by $\ball{x,r}$ (resp.\ $\clball{x,r}$) the open (resp.\ closed) ball of center $x$ and radius $r$ for this norm. We denote by $\Leb$ the $d$-dimensional Lebesgue measure and by $(\base 1, \dots, \base d)$ the canonical basis of $\R^d$.

\subsubsection{Poisson point processes}

Given a locally compact, second countable, Hausdorff topological space $\bbX$, we denote by $\ProSpace[\bbX]$ the space of locally finite Borel measures on $\bbX$ with values in $\N\cup \acc\infty$, endowed with the $\sigma$-algebra induced by evaluations on Borel subsets. Let $\mu$ be a locally finite measure on $\bbX$. We say that a random variable $\progen$ with values in $\ProSpace[\bbX]$ is a \emph{Poisson point process} with \emph{intensity} $\mu$ if:
\begin{enumerate}[(i)]
	\item For all Borel subsets $B\subseteq \bbX$, $\progen(B)$ follows a Poisson distribution with parameter $\mu(B)$, with the convention that a Poisson variable with infinite parameter is infinite almost surely.
	\item For all families $(B_1, \dots, B_n)$ of pairwise disjoint Borel subsets of $\bbX$, the variables $\progen(B_1)$,..., $\progen(B_n)$ are independent.
\end{enumerate}
We refer to Baccelli-Blaszczyszyn-Karray \cite{Bac20} for a detailed review of Poisson point processes, and point processes in general. Proposition 1.6.11 there states that for all measure $\mu \in \ProSpace[\bbX]$, there exists $N\in \N\cup\acc\infty$ and points $(x_n)_{1\le n \le N}$ in $\bbX$ such that
\begin{equation}
	\label{eqn : Intro/Sum_of_Diracs_General}
	\mu = \sum_{n=1}^N \Dirac{x_n}.
\end{equation}
Moreover, $N$ and $(x_n)_{1\le n \le N}$ may be chosen in a measurable way with respect to $\mu$.

In this article, we fix a nontrivial, locally finite measure $\nu$ on $\intervalleoo0\infty$ and a Poisson point process $\pro$ on $\R^d \times \intervalleoo0\infty$, with intensity $\Leb \otimes \nu$. One easily checks that $\pro$ is stationary, in the sense that for all $z\in \R^d$, $\pro$ has the same distribution as its image through $(x,t)\mapsto (x+z,t)$. Moreover, by Proposition~2.1.9 in \cite{Bac20}, $\pro$ is simple as a marked point process on $\R^d$, i.e.\ almost surely, for all $x\in \R^d$,
\begin{equation*}
	\pro\p{ \acc x \times \intervalleoo0\infty } \in \acc{0,1}.
\end{equation*}
Consequently,~\eqref{eqn : Intro/Sum_of_Diracs_General} for $\mu = \pro$ takes the form
\begin{equation}
	\pro = \sum_{n=1}^\infty \Dirac{\p{ z_n, \Mass{z_n} }},
\end{equation}
where $\p{ z_n, \Mass{z_n} }_{n\ge 1}$ is a random sequence with values in $\R^d \times \intervalleoo0\infty$.

For every Borel subset $B\subseteq \R^d$, we call \emph{mass} of $B$ the random variable
\begin{equation}
	\label{eqn : Intro/Formal/MassA}
	\Mass B \dpe \int_{B \times \intervalleoo0\infty} t\pro(\d x, \d t) = \sum_{n=1}^\infty \Mass{z_n}\ind{z_n \in B} .
\end{equation}
For all $x\in \R^d$, we will write $\Mass x = \Mass{\acc{x}}$, which causes no notation collision. We call \emph{support} of $\pro$ the random set
\begin{equation}
	\projpro \dpe \set{x\in \R^d}{\Mass x> 0 } = \set{z_n}{n\ge 1}.
\end{equation}

\subsubsection{Continuous animals and paths}

\begin{Definition}
	\label{def : Intro/Formal/ContinuousPath}
	We call \emph{(continuous) path} a finite sequence of points in $\R^d$. The \emph{length} of the path $\gamma = (x_0, \dots, x_r)$ is defined as
	\begin{equation}
		\norme \gamma \dpe \sum_{i=0}^{r-1}\norme{x_i - x_{i+1}}.
	\end{equation}
	The \emph{mass} of $\gamma$ is simply $\Mass \gamma \dpe \Mass{\acc{x_i}_{0\le i \le r} }$. For all $x,y \in \R^d$ and $\ell \ge \norme{x-y}$, we denote by $\SetPUF\ell$ the set of paths starting at $0$, with length at most $\ell$ and by $\SetPDF xy\ell$ the set of paths starting at $x$ and ending at $y$, with length at most $\ell$. We define the random variables
	\begin{equation}
		\label{def : Intro/Formal/ContinuousAnimal/MassPUF_PDF}
		\MassPUF\ell \dpe \sup_{\gamma \in \SetPUF\ell}\Mass\gamma \text{ and } \MassPDF xy\ell \dpe \sup_{\gamma \in \SetPDF xy\ell}\Mass\gamma.
	\end{equation}
\end{Definition}
\begin{Definition}
	\label{def : Intro/Formal/ContinuousAnimal}
	We call \emph{(continuous) animal} a connected, finite graph $\xi=(V,E)$, with $V\subseteq \R^d$. Given an animal $\xi = (V,E)$ we define its \emph{length} as
	\begin{equation}
		\norme \xi \dpe \sum_{\acc{x,y} \in E}\norme{x-y}
	\end{equation}
	and its \emph{mass} as $\Mass \xi \dpe \Mass V$. For all $x,y \in \R^d$ and $\ell\ge \norme{x-y}$, we denote by $\SetAUF\ell$ the set of animals containing $0$, with length at most $\ell$ and by $\SetADF xy\ell$ the set of animals containing $x$ and $y$, with length at most $\ell$. We define the random variables
	\begin{equation}
		\label{def : Intro/Formal/ContinuousAnimal/MassAUF_ADF}
		\MassAUF\ell \dpe \sup_{\xi \in \SetAUF\ell}\Mass\xi \text{ and } \MassADF xy\ell \dpe \sup_{\xi \in \SetADF xy\ell}\Mass\xi.
	\end{equation}
\end{Definition}
Note that any path may be seen as an animal. When possible without creating ambiguity, we will identify a path or an animal with its vertex set to simplify notations. Note that the supremum in~\eqref{def : Intro/Formal/ContinuousAnimal/MassPUF_PDF} may be taken over paths included in $\projpro \cup \acc0$ (resp. $\projpro\cup \acc{x,y}$ ), since removing a massless point from a path has no effect on the path's mass and decreases its length. This does not hold in~\eqref{def : Intro/Formal/ContinuousAnimal/MassAUF_ADF}, hence we get a possibly different model by restricting ourselves to animals included in $\projpro \cup \acc0$ (resp. $\projpro\cup \acc{x,y}$) or, more generally by associating a negative mass $-q$ to points outside $\projpro \cup \acc 0$ (resp. $\projpro\cup \acc{x,y}$), for a fixed $q\in\intervalleff0\infty$. However, to make concatenation arguments simpler we use the following definition.
\begin{Definition}
	\label{def : Intro/Formal/ContinuousAnimalPen}
	Let $x,y \in \R^d$ and $\ell \ge \norme{x-y}$. We denote by $\SetAUFalt\ell$ the set of animals $\xi=(V,E)$ either empty, or satisfying
	\begin{equation}
		\norme\xi + \d (0, V) \le \ell,
	\end{equation}
	by $\SetADFalt xy\ell$ the set of animals $\xi$ either empty, or satisfying
	\begin{equation}
		\norme\xi + \d (x, V) + \d (y,V) \le \ell.
	\end{equation}
	For all $q\in\intervalleff0\infty$, we define
	\begin{equation}
		\label{def : Intro/Formal/ContinuousAnimalPen/MassAUF_ADF}
		\MassAUFpen\ell q \dpe \sup_{\xi \in \SetAUFalt\ell}\Masspen\xi \text{ and } \MassADFpen xy\ell q \dpe \sup_{\xi \in \SetADFalt xy\ell}\Masspen\xi,
	\end{equation}
	where
	\begin{equation}
		\Masspen B \dpe \Mass B - q\#\p{ B \cap (\R^d \setminus \projpro ) },
	\end{equation}
	and $\#$ denotes the cardinal of a set.
\end{Definition}
Note that adding the vertex $0$ and one edge to a nonempty animal $\xi \in \SetAUFalt \ell$ creates an animal in $\SetAUF \ell$. Besides, $\SetAUF \ell \subseteq \SetAUFalt \ell$, therefore $\MassAUFpen\ell 0 = \MassAUF\ell$. Similarly, $\MassADFpen xy\ell 0 = \MassADF xy\ell$. Furthermore, we have the inequalities
\begin{equation}
	\label{eqn : Intro/Formal/BorneGrossiere}
	\MassPUF \ell \le \MassAUFpen \ell q \le \MassAUF \ell \le \MassPUF{2\ell}.
\end{equation}
Only the last one is not direct. It is a consequence of the fact the path obtained by depth-first search of an animal is at most twice as long as the animal. To shorten the statement of results and some proofs, we will use $\mathrm G$ to denote either $\mathrm P$ or $\mathrm A^{(q)}$, for $q\in \intervalleff0\infty$. Similarly, the letter $\mathcal G$ will denote either $\mathcal P$ or $\mathcal A^*$, and $\mathbf G$ will denote either $\mathbf P$ or $\mathbf A^{(q)}$ (see Theorem~\ref{thm : Intro/LLN}).

We can now state a weak version of Theorem~1.6 in \cite{Article3} that will be sufficient for our purposes. It may be seen as the analogue of~\eqref{eqn : Intro/Background/LLN} in the continuous framework. 
\begin{Theorem}
	\label{thm : Intro/LLN}
	Assume that $\nu$ satisfies~\eqref{ass : Greedy}. Then there exists a deterministic, concave, decreasing, continuous function $\LimMassG : \intervalleff01 \rightarrow \intervallefo0\infty$ such that for all $\beta \in \intervalleff01$,
	\begin{equation}
		\label{eqn : Intro/LLN/Directed}
		\frac{\MassGDF{0}{\ell \beta \base 1}{\ell} }{\ell} \xrightarrow[\ell \to\infty]{\text{a.s.\ and }\rL^1} \LimMassG(\beta).
	\end{equation}
	Moreover,
	\begin{equation}
		\label{eqn : Intro/LLN/Undirected}
		\frac{\MassGUF{\ell} }{\ell} \xrightarrow[\ell \to\infty]{\text{a.s.\ and }\rL^1} \LimMassG(0)
	\end{equation}
	and
	\begin{equation}
		\label{eqn : Intro/LLN_Animal/GeneralUB}
		\sup_{\ell >0} \frac{\E{\MassAUF \ell} }\ell \le \Cl{GREEDY}\int_0^\infty \nu\p{\intervallefo t\infty}^{1/d} \d t,
	\end{equation}
	where $\Cr{GREEDY}$ is a constant that only depends on the dimension.
\end{Theorem}
Along the proof of Theorem~\ref{thm : Intro/LLN}, an analogue for an auxiliary process called the diamond process was shown; this analogue will also play a central role here (see Section~\ref{subsec : Preli/AuxProcesses}). 

\subsection{Main results}
\label{sec : Intro/Main}

Our first result, Theorem~\ref{thm : MAIN/LowerTail}, gives a bound for the lower-tail deviation probability.
\begin{Theorem}
	\label{thm : MAIN/LowerTail}
	Assume that $\nu$ satisfies~\eqref{ass : Greedy}. For all $\beta \in \intervalleoo01$, $\zeta < \LimMassG(\beta)$,
	\begin{equation}
		\label{eqn : MAIN/LowerTail/Directed}
		\liminf_{\ell \to\infty} - \frac{1}{\ell^d} \log \Pb{ \MassGDF{0}{\ell \beta \base 1}\ell \le \zeta \ell } > 0
	\end{equation}
	and for all $\zeta < \LimMassG(0)$,
	\begin{equation}
		\label{eqn : MAIN/LowerTail/Undirected}
		\liminf_{\ell \to\infty} - \frac{1}{\ell^d} \log \Pb{ \MassGUF\ell \le \zeta \ell } > 0.
	\end{equation}
\end{Theorem}
Under a stronger assumption, Theorem~\ref{thm : MAIN/UpperTail} gives the order and the existence of the rate function for the upper-tail deviation probability.
\begin{Theorem}
	\label{thm : MAIN/UpperTail}
	Assume that there exists $\lambda >0$ such that
	\begin{equation}
		\label{ass : ExpMoment} \tag{Moment}
		\int_{\intervalleoo0\infty} \p{ \mathrm e^{\lambda t} -1 }\nu(\d t)<\infty.
	\end{equation}
	Then for all $(\beta,\zeta) \in \intervallefo01 \times \intervalleoo0\infty$, the limit
	\begin{equation}
		\label{eqn : MAIN/UpperTail/Directed}
		\FdT_\rG(\beta,\zeta)  \dpe \lim_{n \to \infty} -\frac1n \log \Pb{ \MassGDF{0}{\ell \beta \base 1}\ell \ge \zeta \ell }
	\end{equation}
	exists and it is finite. Moreover:
	\begin{enumerate}[(i)]
		\item \label{item : MAIN/UpperTail/Undirected} For all $\zeta \in \intervalleoo0\infty$,
			\begin{equation}
				\label{eqn : MAIN/UpperTail/Undirected}
				\lim_{n\to \infty} -\frac1n \log \Pb{\MassGUF\ell \ge \zeta \ell } = \FdT_\rG(0,\zeta).
			\end{equation}
		\item \label{item : MAIN/UpperTail/ConvexNondec} The function $\FdT_\rG$ is continuous, convex and nondecreasing on $\intervallefo01 \times \intervalleoo0\infty$.
		\item \label{item : MAIN/UpperTail/OdG} For all $(\beta,\zeta) \in \intervallefo01 \times \intervalleoo0\infty$, \[ \FdT_\rG(\beta,\zeta) =0 \text{ if and only if } \zeta \le \LimMassG(\beta).\]
		\item \label{item : MAIN/UpperTail/Limit} For all $\beta \in\intervallefo01$, $\FdT_\rG(\beta, \cdot)$ is strictly increasing on $\intervallefo{\LimMassG(\beta)}{\infty}$, and
			\begin{equation}
				\lim_{\zeta \to \infty} \FdT_\rG(\beta, \zeta) = \infty.
			\end{equation}
	\end{enumerate}
\end{Theorem}
By classic large deviation arguments detailed in Appendix~\ref{appsec : LDP}, Theorems~\ref{thm : MAIN/LowerTail} and~\ref{thm : MAIN/UpperTail} imply Corollary~\ref{cor : LDP}.
\begin{Corollary}
 	\label{cor : LDP}
 	Assume that~\eqref{ass : ExpMoment} holds. For all $(\beta,\zeta) \in \intervallefo01 \times \intervallefo0\infty$, define
 	\begin{equation}
 		\FdT_\rG^*(\beta, \zeta) \dpe \begin{cases}
 			\FdT_\rG(\beta, \zeta) &\text{ if } \zeta \ge \LimMassG(\beta),\\ \infty &\text{ if } \zeta < \LimMassG(\beta).
 		\end{cases}
 	\end{equation}
 	Let $\beta \in \intervallefo01$. Then $\p{ \frac1\ell\MassGDF{0}{\ell \beta \base 1}{\ell} }_{\ell >0}$ follows the large deviation principle (LDP) at speed $\ell$, with the rate function $\FdT_\rG^*(\beta, \cdot)$. In other words, for every Borel subset $B \subseteq \intervallefo0\infty$, denoting by $\overline B$ its closure and by $\mathring B$ its interior, we have
 	\begin{equation}
 		\begin{split}
 			\inf_{\zeta \in \overline B} \FdT_\rG^*(\beta, \zeta)%
 				&\le \liminf_{\ell \to \infty}-\frac1\ell \log \Pb{ \frac1\ell \MassGDF{0}{\ell \beta \base 1}{\ell} \in B }\\
 				&\le \limsup_{\ell \to \infty}-\frac1\ell \log \Pb{ \frac1\ell \MassGDF{0}{\ell \beta \base 1}{\ell} \in B }%
 				\le \inf_{\zeta \in \mathring B} \FdT_\rG^*(\beta, \zeta).
 		\end{split}
 	\end{equation}
 	Moreover, $\p{ \frac1\ell \MassGUF{\ell} }_{\ell >0}$ follows the large deviation principle at speed $\ell$, with the rate function $\FdT_\rG^*(0, \cdot)$.
\end{Corollary} 

Along the proof of Theorem~\ref{thm : MAIN/UpperTail}\eqref{item : MAIN/UpperTail/OdG}, we will actually prove some explicit lower bounds for $\FdT_\rG$, under some additional assumptions on $\nu$.
\begin{Theorem}
	\label{thm : BOUNDS}
	Assume that $\nu$ is supported by $\intervalleof0m$, with $0<m<\infty$, and $\mathrm G = \mathrm P$ or $\mathrm A$. Then for all $\beta \in \intervallefo01$ and $\zeta > \LimMassG(\beta)$,
	\begin{equation}
		\label{eqn : BOUNDS/Unpen}
		\FdT_\rG(\beta,\zeta) \ge h\p{ \frac{\zeta - \LimMassG(\beta)}{ \LimMassG(\beta) }  } \frac{\LimMassG(\beta)}{m},
	\end{equation}
	where $h(s) = (1+s)\log(1+s) - s$. Further assume that $\nu$ is finite. Let $q\in \intervallefo0\infty$. Then for all $\beta \in \intervallefo01$ and $\zeta > \LimMassG(\beta)$,
	\begin{equation}
		\label{eqn : BOUNDS/Pen}
		\FdT_{\rA^{(q)}}(\beta,\zeta) \ge \max_{\alpha > 0} \min \p{ h(\alpha) \Cr{DIRAC} \nu\p{\intervalleof0m}^{1/d} , \frac{ \p{ \zeta - \LimMassA^{(q)}(\beta)}^2 }{(1+\alpha)  \Cr{DIRAC} \nu\p{\intervalleof0m}^{1/d} (q+m)^2 } },
	\end{equation}
	where $C_2>0$ only depends on the dimension.
\end{Theorem}
\begin{Remark}
	Fix $\beta \in \intervallefo01$. 
	The right-hand side $f(\zeta)$ of~\eqref{eqn : BOUNDS/Unpen} satisfies
		\begin{equation}
			f(\zeta) \underset{\zeta \to \LimMassG(\beta)}\sim \frac{\p{\zeta - \LimMassG(\beta)}^2 }{ 2 \LimMassG(\beta) m} \text{ and } f(\zeta) \underset{\zeta \to \infty}\sim \frac{\zeta \log \zeta}{m},
		\end{equation}
	while the right-hand side $g(\zeta)$ of~\eqref{eqn : BOUNDS/Pen} satisfies
		\begin{equation}
			g(\zeta) \underset{\zeta \to \LimMassG(\beta)}\sim \frac{\p{\zeta - \LimMassG(\beta)}^2 }{ \Cr{DIRAC} \nu\p{\intervalleof0m}^{1/d} (q+m)^2 }  \text{ and } g(\zeta) \underset{\zeta \to \infty}\sim \frac{\zeta (\log \zeta)^{1/2}}{\Cr{DIRAC} \nu\p{\intervalleof0m}^{1/d} (q+m)  }. \end{equation}
	Note that $f(\zeta)$ and $g(\zeta)$ are both of order $\p{\zeta - \LimMassG(\beta)}^2$ for $\zeta$ close to $\LimMassG(\beta)$, and only differ by a $\log(\zeta)^{1/2}$ factor as $\zeta \to \infty$.
\end{Remark}


\subsection{Organization of the paper and outline of the proofs}
\label{subsec : Intro/OotP}

Section~\ref{sec : Preli} introduces the diamond processes which are similar to $\MassGDF\cdot\cdot\cdot$ and already appeared in \cite{Article2}, and two technical lemmas providing inequalities between the maximal mass harvested by animals belonging to certain sets, assuming these sets are linked by concatenation relations.

Section~\ref{sec : LowerTail} contains the proof of Theorem~\ref{thm : MAIN/LowerTail}. The core idea is to show that on the lower-tail deviation event, there exists a set of $\Theta(\ell^{d-1})$ pairwise disjoint corridors, each of which exhibiting lower-tail behaviour.

Section~\ref{sec : Concentration} contains some concentration bounds for $\MassGDF\cdot\cdot\cdot$. We use general concentration results on functions satisfying the self-bounding or the bounded difference property, which are known for independent vector of random variables, and adapted here for Poisson point processes.  

Section~\ref{sec : UpperTail} is devoted to the proof of Theorem~\ref{thm : MAIN/UpperTail}. A subadditivity argument gives the equivalent of~\eqref{eqn : MAIN/UpperTail/Directed} for diamond processes. We then use elementary concatenation arguments to get~\eqref{eqn : MAIN/UpperTail/Directed}. To get the right order for the upper-tail large deviation probability~\eqref{item : MAIN/UpperTail/OdG}, we adapt the strategy of Dembo-Gandolfi-Kesten (see~\cite{Dem01}, Theorem~4.1): we truncate the process $\pro$, then control the bounded part with Theorem~\ref{thm : BOUNDS} (which is essentially an asymptotic version of the bounds from Section~\ref{sec : Concentration}) and the unbounded one by a Chernoff-like bound.

\subsection{Related works and open questions}
\label{subsec : RWOQ}

\paragraph{Last-passage percolation.} Last-passage percolation may be seen as a variant of the greedy model, where paths are required to be oriented. More precisely, in dimension $2$, consider a vector $(X_v)_{v\in \Z^2}$ of i.i.d.\ nonnegative variables, and define \[ L_n \dpe \max_{(0,0) \Path{\gamma} (n,n) } \sum_{v\in \gamma} X_v,  \] where the maximum spans over the set of lattice paths from $(0,0)$ to $(n,n)$ that contain only rightward and upward steps. Provided integrability, a routine subadditive argument shows that $L_n/n$ converges to a deterministic constant $\mathbf L$, and for all $\eps > 0$, the limit \[ \lim_{n\to \infty} -\frac1n \log\Pb{L_n \ge (\mathbf L + \eps) n} \] exists and is positive. Adapting the proof of (5.13) in \cite{KestenStFlour}, one shows that the lower-tail large deviation probability has order $\exp(-n^2)$. In the solvable case where the $(X_v)_{v\in\Z^2}$ follows the exponential distribution with parameter $1$, Logan-Shepp \cite{Log77} and Vershik-Kerov \cite{Ver77} independently proved that $\mathbf L =4$. Baik-Deift-Johansson \cite{Bai99} showed that the scaling limit of $L_n$ is the Tracy-Widom distribution. Johansson \cite{Joh20} gave explicit rate functions for both the upper-tail and lower-tail deviations of $L_n$ in this context. In the Poissonian framework, which is also solvable, similar large deviation results were obtained by Deuschel-Zeitouni \cite{Deu99} and Seppäläinen \cite{Sep98}.

\paragraph{First-passage percolation.} First-passage percolation on the $\Z^d$ lattice consists in endowing each edge with a random positive passage time, such that the family of the passage times are i.i.d., and considering the associated metric $\mathrm T$ on $\Z^d$. It is known that under mild moment conditions, $\mathrm T(0,n\base 1)/n$ converges a.s.\ to a deterministic constant $\mu$, called the \emph{time constant}. Moreover, Kesten \cite[Theorem 5.2]{KestenStFlour} showed that the lower-tail large deviation probability for $\mathrm T(0,n\base 1)$ has order $\exp(-n)$, along with the existence and convexity of the corresponding rate function. Theorem 5.9 in the same paper states that if the edge passage times are bounded a.s., then the upper-tail large deviation probability for $\mathrm T(0,n\base 1)$ has order $\exp(-n^d)$. More than thirty years later, under an additional regularity condition and for $d=2$,\footnote{The proof can be adapted for higher dimensions, though.} Basu, Ganguly and Sly \cite{Bas21} proved the existence of the rate function in this regime, i.e.\ of the limit 
\begin{equation}\label{eqn : Intro/RWOQ/BGS} 
	\lim_{n\to \infty} -\frac1{n^2} \log \Pb{\mathrm T(0 , n\base 1) \le (\mu - \eps)n}.
\end{equation}
The author recently proved \cite{Article4, Article1} two large deviation principles, at speed $n$ and $n^d$ respectively, for the random metric, thus extending the results from \cite{Bas21} and \cite{KestenStFlour}.

\paragraph{Maximal flow.} An edge's weight in first-passage percolation may also be interpreted as the maximal flow of some liquid along this edge. With this point of view, a quantity of interest is the maximal flow $\phi_n$ admissible by the network across the box $\intint0n^d$. The a.s.\ convergence of $\phi_n/ n^{d-1}$ to a deterministic constant was proven by Kesten \cite{Kes87} in dimension $3$ and Zhang \cite{Zha18} in higher dimension. Cerf and Théret \cite{CT11b, CT11c} showed that the lower-tail and upper-tail deviation probabilities of $\phi_n$ have order $\exp\p{-n^{d-1}}$ and $\exp\p{-n^{d}}$ respectively. Dembin and Théret \cite{Dem21+_Surface, Dem21+_Volume} recently proved the existence of both upper-tail and lower-tail rate functions for $\phi_n$. In the former regime, it is a consequence of a large deviation principle at speed $n^d$ that gives estimates for the probability that a given mesoscopic stream --- i.e.\ a vector field that describes one way the liquid could travel through the medium without breaking conservation of mass --- is admissible in $\intint0n^d$.

\paragraph{Lower-tail for greedy animals and paths.} Given Theorem~\ref{thm : MAIN/LowerTail}, it is natural to wonder whether for all $\zeta < \LimMassG(0)$, the limit
\begin{equation}
	\label{eqn : Intro/RWOQ/LDP_LowerTail?}
	\lim_{\ell \to\infty} - \frac{1}{\ell^d} \log \Pb{ \MassGUF\ell \le \zeta \ell }.
\end{equation}
exists. One approach could be to find a way to assemble $k^d$ independent realizations of $\acc{ \MassGUF\ell \le \zeta \ell }$ on size $\ell$ boxes into a realization of a size $k\ell$ version of this event, and use a multiparameter subadditive argument. The main obstacle would be that such a construction is not trivial. In particular, naively sticking boxes next to each other does not produce the desired event. The authors of \cite{Bas21}, \cite{Dem21+_Volume} and \cite{Article1} all encountered this difficulty in different models; they relied in some shape to a description of the large box as a mosaic of nearly uniform tiles, and apply the multiparameter subadditive argument to those tiles. Such an approach --- we would like to call it the \emph{tile method} --- may very well be fruitful for greedy animals and paths but would require an understanding of the natural object that contains all the information about masses gathered by greedy animals and paths, the same way the random metric in first-passage percolation contains all the information about the passage time between any two subsets of $\Z^d$.

\paragraph{Extension of Theorems~\ref{thm : MAIN/LowerTail} and~\ref{thm : MAIN/LowerTail} to other greedy models.} All our proofs rely on independence, some of them for multiple reasons. Notably, the concentration bounds we use fail for dependent fields, as well as the fundamental Lemmas \ref{lem : Concatenation} and \ref{lem : Decomposition}. Rotation-invariance, on the contrary, is not crucial: for example, we believe our main results still hold in the discrete setting. The main changes are the following: \begin{itemize}
\item All occurrences of $\beta \base 1$, with $\beta \in \intervallefo01$ should be replaced by a point $u$ in $\ball[1]{0,1}$, the open unit ball for $\norme[1]\cdot$.
\item The rate function $\FdT$ appearing in Theorem~\ref{thm : MAIN/UpperTail} should be defined on $\ball[1]{0,1} \times \intervalleoo0\infty$. 
\item The constructions of Sections~\ref{prop : UpperTail/Limit} and \ref{subsec : UpperTail/OdG} should be replaced the ones from \cite[Section 2]{Article2}, which are more sophisticated and involve, besides the diamond process, another object called the antidiamond process.
\end{itemize}
We choose not to detail those and stick to the continuous setting for the sake of simplicity.

\paragraph{About the moment assumption.} Assumption~\eqref{ass : ExpMoment} is equivalent to the conjunction of
\begin{align}
 	\label{eqn : ExpMoment 1} \exists \lambda > 0 \text{ s.t.\ } \int_1^\infty \mathrm e^{\lambda t} \nu(\d t) &< \infty, \intertext{ and } \label{eqn : ExpMoment 2} \int_0^1 t\nu(\d t) &< \infty.
\end{align} 

Theorem~\ref{thm : MAIN/UpperTail} fails without \eqref{eqn : ExpMoment 1}. Indeed, let $\nu$ be a probability measure whose exponential moments are all infinite. For all $i \in \N$ let $X_i$ denote the maximal mass of atoms in $\intervalleff01^d + i\base 1$. The exponential moments of $X_0$ are also infinite, therefore by Cramér's theorem (see e.g.\ Theorem 2.2.3 in \cite{LDTA}), for all $\zeta >0$,
\[ \lim_{n\to \infty} -\frac1n \log \Pb{\sum_{i=0}^{n-1} X_i \ge \zeta n } =0. \]
Moreover, $\sum_{i=0}^{n-1} X_i \le \MassPUF{C n}$, where $C$ only depends on the dimension, thus for all $\zeta >0$
\[ 
\lim_{\ell\to \infty} -\frac1\ell \log \Pb{\MassPUF\ell \ge \zeta \ell } =0.
\]

We do not know if~\eqref{eqn : ExpMoment 2}, which is equivalent to the expected mass of $\intervalleff01^d$ for $\restriction\pro{\R^d \times \intervalleof01}$ being finite, is optimal.

\subsection{Notations}
\label{subsec : Intro/Notations}

\subsubsection{Rate functions}

To lighten the notations, the rate functions defined in Theorem~\ref{thm : MAIN/UpperTail} and Corollary~\ref{cor : LDP} will be written as $\FdT$ and $\FdT^*$ instead of $\FdT_\rG$ and $\FdT_\rG^*$.

\subsubsection{Point processes}
Given a locally compact, second countable, Hausdorff topological space $\bbX$, we denote by $\ProSpaceFinite[\bbX]$ the subspace of $\ProSpace[\bbX]$ consisting of finite measures, and $\DiracSpace[\bbX]$ the subspace of $\ProSpace[\bbX]$ consisting of measures with mass $0$ or $1$, i.e.\ the trivial measure and Dirac masses on $\bbX$.

If $X$ is a mesurable function of the Poisson process $\pro$ and $\pro'$ is another point process on $\R^d \times \intervalleoo0\infty$, we denote by $X\cro{\pro'}$ the random variable obtained when one  replaces $\pro$ by $\pro'$ in the definition of $X$. We will use the same convention to indicate the dependence with respect to $\nu$.  

\subsubsection{Concatenation of animals and paths}

Consider two animals $\xi_1 = (V_1, E_1) $ and $\xi_2 = (V_2, E_2)$. We define the \emph{concatenation} of $\xi_1$ and $\xi_2$ as any minimizer of $\norme \xi$, among animals $\xi=(V,E)$ that contains both $\xi_1$ and $\xi_2$ as subgraphs. In case of nonunicity, it is chosen according to a deterministic rule. Note the following properties:
\begin{enumerate}[(i)]
	\item If $V_1 \cap V_2 \neq \emptyset$, then $(V,E) = \p{V_1\cup V_2, E_1 \cup E_2}$.
	\item If $V_1$ and $V_2$ are both nonempty, then
		\begin{equation}
			\norme{\xi_1 \concat \xi_2} \le \norme{\xi_1} + \norme{\xi_2} + \d(\xi_1, \xi_2).
		\end{equation}
	\item If $\xi_1$ and $\xi_2$ are paths that share at least one endpoint, then $\xi_1 \concat \xi_2$ is a path.
\end{enumerate}
For all animals $\xi_1, \dots, \xi_n$, we define
\[
	\xi_1 \concat \xi_2 \concat \dots \concat \xi_n \dpe (((\dots ( \xi_1 \concat \xi_2) \concat \xi_3 ) \concat \dots \concat \xi_n ).
\]


\section{Technical preliminaries}
\label{sec : Preli}

In this section we introduce the diamond processes, namely the maximal mass of animals included in a diamond and visiting both extremal points of the diamond. Contrary to $\MassGDF\cdot \cdot \cdot$, they display some superadditive properties. We then present some general tools to compare the maximal masses of animals in given sets.

\subsection{Diamond animals}
\label{subsec : Preli/AuxProcesses}

For all $x\in \R^d$, $u\in \R^d\setminus \acc 0$ and $0<\delta \le 1$, we define the \emph{cone}
\begin{equation}
  \label{eqn : intro/notations/cone}
  \Cone{\delta}{x}{u} \dpe \set{z\in \R^d }{\ps{z-x}{\frac{u}{\norme[2]{u} }} \ge (1-\delta)\norme[2]{z-x} }.
\end{equation}
For all distinct points $x,y\in \R^d$ and $0<\delta \le 1$, we define the \emph{diamond}
\begin{equation}
  \label{eqn : intro/notations/diamond}
  \Diamant{\delta}{x}{y} \dpe \Cone{\delta}{x}{y-x} \cap \Cone{\delta}{y}{x-y} 
\end{equation}
Note that for $\delta =1$, this is exactly the domain located between the two hyperplanes orthogonal to $x-y$ that contains $x$ and $y$ and respectively.

For all $0< \delta \le 1$, $\ell>0$ and distinct points $x,y\in \R^d$, we define
\begin{equation}
	\SetGDC{\delta}{x}{y}{\ell} \dpe \set{\xi \in \SetGDF{x}{y}{\ell} }{\xi \subseteq \Diamant{\delta}{x}{y}}
\end{equation}
We say that an animal (resp. path) is a \emph{cylinder animal} (resp.\ \emph{cylinder path}) if it belongs to $\SetADCalt{1}{x}{y}{\ell}$ (resp.\ $\SetPDC{1}{x}{y}{\ell}$), for some $x$, $y$ and $\ell$. The diamond process $\MassGDC{\delta}{x}{y}{\ell}$ is defined as the maximal (maybe penalized) mass over animals or paths in this set.
Proposition~\ref{prop : Preli/LLN} states that the law of large numbers~\eqref{eqn : Intro/LLN/Directed} also holds for the diamond process, with the same limit. It is a consequence of Lemmas~2.2 to~2.4 in \cite{Article3}. 
\begin{Proposition}
	\label{prop : Preli/LLN}
	Let $0<\delta\le 1$. For all $\beta \in \intervalleoo01$,
	\begin{equation}
		\label{eqn : Preli/LLN/GDC}
		\frac{\MassGDC{\delta}{0}{\ell \beta \base 1}{\ell} }{\ell } \xrightarrow[\ell \to \infty]{\text{a.s.\ and }\rL^1} \LimMassG(\beta).
	\end{equation}
\end{Proposition}

\subsection{The concatenation and decomposition lemmas}
This section's main results, Lemmas~\ref{lem : Concatenation} and~\ref{lem : Decomposition}, encapsulate all our construction arguments, as they provide some inequalities between the maximal animal masses, over different sets of animals, assuming these sets satisfy certain concatenation relations.
\begin{Definition}
	\label{def : Concatenation/Concatenables}
	Let $\mathcal G_1,\dots,\mathcal G_k$ and $\mathcal G$ be collections of animals. We say that the family $(\mathcal G_1,\dots,\mathcal G_k)$ is \emph{concatenable in} $\cG$ if there exists a Lebesgue-negligible set $B\subseteq \R^d$ such that for all $\xi_1 \in \mathcal G_1,\dots,\xi_k \in \mathcal G_k$,
	\begin{equation} \label{eqn : Concatenation/concat} \xi_1 \concat \dots \concat \xi_k \in \mathcal G \end{equation} 
	and \begin{equation} \label{eqn : Concatenation/intersection} \forall 1\le i < j \le k,\quad \xi_i \cap \xi_j \subseteq B.\end{equation}
\end{Definition}
\begin{Lemma}
	\label{lem : Concatenation}
	Let $q \in \intervalleff0\infty$. Let $\mathcal G_1$,...,$\mathcal G_k$ and $\mathcal G$ be collections of animals. Assume that $(\mathcal G_1,\dots,\mathcal G_k)$ is \emph{concatenable in} $\cG$. Let us denote
	\begin{equation}
		\rG^{(q)} \dpe \sup_{\xi \in \cG} \Masspen \xi \text{ and } \rG_{i}^{(q)} \dpe \sup_{\xi \in \cG_{i}} \Masspen \xi. 	
	\end{equation}
	Then almost surely,
	\begin{equation}
		\label{eqn : Concatenation/Inequality}
	   	\sum_{i=1}^k \rG_i^{(q)}  \le \rG^{(q)},
	\end{equation}
	and the terms of the left-hand side are independent.
\end{Lemma}
\begin{proof}
	Let us assume that the almost sure event $\Mass{B}=0$ occurs. Let $\xi_1 \in \mathcal G_1,\dots, \xi_k \in \mathcal G_k$. By~\eqref{eqn : Concatenation/intersection},
	\[ \sum_{i=1}^k \Mass{\xi_i} = \Mass{ \bigcup_{i=1}^k \xi_i }. \]
	Besides, $\# \p{ (\xi_1\cap \xi_2)\cap \p{\R^d \setminus \projpro} } \le \# \p{ \xi_1\cap \p{\R^d \setminus \projpro} } + \# \p{\xi_2\cap \p{\R^d \setminus \projpro} }$, thus
	\[ \sum_{i=1}^k \Masspen{\xi_i} \le \Masspen{ \bigcup_{i=1}^k \xi_i }. \]
	By~\eqref{eqn : Concatenation/concat},
	\[ \sum_{i=1}^k \Masspen{\xi_i} \le \sup_{\xi \in \mathcal G} \Masspen{\xi}, \]
	thus taking the supremum with respect to $(\xi_i)_{1\le i \le k}$ yields~\eqref{eqn : Concatenation/Inequality}. The independence comes from the fact that the terms of the sum live in pairwise disjoint subsets of $\R^d$, and $\pro$ is a Poisson process.
\end{proof}
\begin{Definition}
	\label{def : Concatenation/Decomposition}
	Let $\p{ \mathcal G_{n,i} }_{ \substack{ 1 \le n \le N \\ 1\le i \le I_n}}$ be a finite family of collections of animals of $\R^d$. Let $\cG$ be another collection of animals of $\R^d$. We say that $\cG$ is \emph{decomposable} in $\p{ \mathcal G_{n,i} }$ if for all $\xi \in \cG$, there exists $1\le n \le N$ and $\xi_1 \in \cG_{n,1}, \dots, \xi_{I_n} \in \cG_{n, I_n}$, such that
	\begin{equation}
		\xi \subseteq \bigcup_{i=1}^{I_n} \xi_i.
	\end{equation}
\end{Definition}
\begin{Lemma}
	\label{lem : Decomposition}
	Let $\cG$ and $\p{ \mathcal G_{n,i} }_{ \substack{ 1 \le n \le N \\ 1\le i \le I_n} }$ be as in Definition~\ref{def : Concatenation/Decomposition}. Assume that there exists $r>0$ such that 
	\[
		\forall \xi \in \cG \cup \p{ \bigcup_{\substack{ 1 \le n \le N \\ 1\le i \le I_n}} \mathcal G_{n,i} }, \quad \xi \in \clball{0,r}.
	\]
	Let us denote
	\begin{equation}
		\label{eqn : Decomposition/Bigcup}
		\rG \dpe \sup_{\xi \in \cG} \Mass \xi \text{ and } \rG_{n,i} \dpe \sup_{\xi \in \cG_{n,i}} \Mass \xi. 	
	\end{equation}
	Then for all $t>0$,
	\begin{equation}
		\label{eqn : Decomposition/Main_inequality}
		\Pb{\rG \ge t} \le \sum_{n=1}^N \sum_{\substack{ (t_i) \in \N^{I_n} \\ t-I_n \le \sum t_i \le t} } \prod_{i=1}^{I_n} \Pb{ \rG_{n,i} \ge t_i }.
	\end{equation}
\end{Lemma}
The proof of Lemma~\ref{lem : Decomposition} relies on a version of BK inequality for Poisson point processes shown by Van den Berg \cite{vdBer96}. We first need to introduce the \emph{disjoint occurence} in this context. We say that a subset $F$ of $\ProSpaceFinite$ is \emph{increasing} if \[ \forall \eta_1, \eta_2 \in \ProSpaceFinite, \quad \p{ \eta_1 \le \eta_2 \text{ and } \eta_1 \in F } \implies \eta_2 \in F. \] 
Given mesurable, increasing subsets $F_1,\dots, F_m$ of $\ProSpaceFinite$, we define the \emph{disjoint occurence} of $F_1, \dots, F_m$ as
\begin{equation}
	F_1 \circ \dots \circ F_m \dpe \set{ \sum_{i=1}^m \eta_i}{ (\eta_1,\dots, \eta_m) \in \prod_{i=1}^m F_i }.
\end{equation}
The main result in \cite{vdBer96} implies that if $\eta$ is a simple Poisson point processes on $\R^d \times \intervalleoo0\infty$ with finite intensity, then
\begin{equation}
	\label{eqn : Decomposition/BK}
	\Pb{ \eta \in \p{ F_1 \circ \dots \circ F_m  } } \le \prod_{i=1}^m \Pb{\eta \in F_i }.
\end{equation}

\begin{proof}[Proof of Lemma~\ref{lem : Decomposition}]
	By monotone convergence, we may assume that $\nu$ is finite, thus the processes $\rG$ and $\rG_{n,i}$ are mesurable functions of a finite intensity Poisson point process. Let $\eta\in \ProSpaceFinite$. There exists $\xi \in \cG$ such that $\Mass\xi\cro\eta = \rG\cro\eta$. Consider the decomposition~\eqref{eqn : Decomposition/Bigcup}. Let $\p{\eta_1, \dots, \eta_{I_n}}$ be any $I_n$-uple of elements of $\ProSpaceFinite$ such that
	\begin{equation*}
		\eta = \sum_{i=1}^{I_n} \eta_i \text{ and } \Mass\xi\cro\eta = \sum_{i=1}^{I_n} \Mass{\xi_i}\cro{\eta_i}.
	\end{equation*}
	Then we have
	\begin{equation*}
		\rG \cro\eta \le \sum_{i=1}^{I_n} \rG_{n,i}\cro{\eta_i}. 
	\end{equation*}
	Consequently, for all $t>0$,
	\begin{align}
		\acc{\rG \ge t} &\subseteq \bigcup_{n=1}^N \bigcup_{\substack{ (t_i) \in \R^{I_n} \\ \sum t_i =t} } \acc{\rG_{n,1} \ge t_1} \circ \dots \circ \acc{\rG_{n,I_n} \ge t_{I_n}} \eol
		&\subseteq \bigcup_{n=1}^N \bigcup_{\substack{ (t_i) \in \N^{I_n} \\ t-I_n \le \sum t_i \le t} } \acc{\rG_{n,1} \ge t_1} \circ \dots \circ \acc{\rG_{n,I_n} \ge t_{I_n}}.
	\end{align}
	By union bound,
	\begin{equation*}
		\Pb{\rG \ge t} \le \sum_{n=1}^N \sum_{\substack{ (t_i) \in \N^{I_n} \\ t-I_n \le \sum t_i \le t} } \Pb{ \acc{\rG_{n,1} \ge t_1} \circ \dots \circ \acc{\rG_{n,I_n} \ge t_{I_n}} }
	\end{equation*}
	Applying~\eqref{eqn : Decomposition/BK} yields~\eqref{eqn : Decomposition/Main_inequality}.
\end{proof}


\section{Lower-tail large deviation estimate}
\label{sec : LowerTail}

In this section we prove Theorem~\ref{thm : MAIN/LowerTail}. The main ingredient to show~\eqref{eqn : MAIN/LowerTail/Directed} is Lemma~\ref{lem : LowerTail/WinCon}, whose proof is postponed to the end of the section. The estimate~\eqref{eqn : MAIN/LowerTail/Undirected} will easily follow from~\eqref{eqn : MAIN/LowerTail/Directed}. We take some inspiration from the proof of (5.13) in \cite{KestenStFlour}.

\begin{Lemma}
	\label{lem : LowerTail/WinCon}
	Let $\beta \in \intervalleoo01$, $\zeta < \LimMassG(\beta)$ and $\eps \dpe \frac14\p{\LimMassG(\beta) - \zeta}$. Then there exist $0< \delta < 1$ and $L>0$ such that for large enough $\ell$,
	\begin{equation}
		\label{eqn : LowerTail/WinCon/Esperance}
		\frac1\ell \floor{\frac{(1-\delta)\ell }{L}} \E{ \MassGDC{1/2}{0}{ \frac{L \beta \base1}{1-\delta}}{L} } \ge \zeta + \eps,
	\end{equation}
	and we have the almost sure inclusion\footnote{We say that an event $A$ is almost included in an event $B$ if $\Pb{A \cap B^\mathrm{c} }=0$.  }
	\begin{equation}
		\label{eqn : LowerTail/WinCon/Inclusion}
		\begin{split}
		&\acc{\MassGDF{0}{\ell \beta \base 1}{\ell} \le \ell \zeta  } \subseteq \\
			&\quad \bigcap_{k \in \intint{0}{\floor{ \frac{\delta \ell}{4d} }}^{d-1}\times \acc0 } %
			\acc{ \sum_{n=0}^{\floor{\frac{(1-\delta)\ell }{L}} -1} \MassGDC{1/2}{2Lk + \frac{ n L \beta \base 1}{1-\delta} }{ 2Lk + \frac{ (n+1) L \beta \base1}{1-\delta} }{L} \le \ell \zeta }.
		\end{split}
	\end{equation}
\end{Lemma}
\begin{proof}[Proof of Theorem~\ref{thm : MAIN/LowerTail}]
	We first prove~\eqref{eqn : MAIN/LowerTail/Directed}. Let $\beta$, $\zeta$, $\eps$, $\delta$ and $L>0$ be as in Lemma~\ref{lem : LowerTail/WinCon}. Note that the diamonds $\Diamant{1/2}{2Lk + \frac{ n L \beta \base 1}{1-\delta} }{ 2Lk + \frac{ (n+1) L \beta \base 1}{1-\delta} }$ involved in~\eqref{eqn : LowerTail/WinCon/Inclusion} are pairwise disjoint, except maybe at their endpoints, thus the variables \[ \MassGDC{1/2}{2Lk + \frac{ n L \beta \base 1}{1-\delta} }{ 2Lk + \frac{ (n+1) L \beta \base 1}{1-\delta} }{L}, \] for $k \in \intint{0}{\floor{ \frac{\delta \ell}{4d} }}^{d-1}\times \acc0$ and $n\in\intint{0}{\floor{\frac{(1-\delta)\ell }{L}} -1}$, are independent. By stationarity of the model and~\eqref{eqn : LowerTail/WinCon/Inclusion}, for large enough $\ell$,
	\begin{equation}
		\label{eqn : LowerTail/PreuveMAIN/WinCon}
		\Pb{ \MassGDF{0}{\ell \beta \base 1}{\ell} \le \ell \zeta } \le \Pb{ \frac{1}{N}\sum_{n=1}^N G_n \le \frac{ \ell \zeta }{N}    }^{N^{d-1}},
	\end{equation}
	where $N\dpe \floor{\frac{(1-\delta)\ell }{L}}$ and the $G_n$ are i.i.d\ copies of $ \MassGDC{1/2}{0}{ \frac{L \beta \base 1}{1-\delta} }{L}$. Besides, by~\eqref{eqn : LowerTail/WinCon/Esperance} and Chernoff's bound there exists $c = c(\beta, \delta, L)>0$ such that for large enough $\ell > 0$,
	\begin{equation*}
		\Pb{ \frac{1}{N}\sum_{n=1}^N G_n \le \frac{ \ell \zeta }{N}   } \le \exp\p{ -c \ell}.
	\end{equation*}
	Combining this inequality with~\eqref{eqn : LowerTail/PreuveMAIN/WinCon} yields~\eqref{eqn : MAIN/LowerTail/Directed}.

	We now turn to the proof of~\eqref{eqn : MAIN/LowerTail/Undirected}. Fix $\zeta < \LimMassG(0)$. Since $\LimMassG$ is continuous on $\intervalleff01$, there exists $\beta \in \intervalleoo01$ such that $\zeta < \LimMassG(\beta)$. Moreover, for all $\ell>0$,
	\begin{equation*}
		\acc{\MassGUF\ell \le \ell \zeta } \subseteq \acc{ \MassGDF{0}{\ell \beta \base 1 }{\ell} \le \ell \zeta },
	\end{equation*}
	thus~\eqref{eqn : MAIN/LowerTail/Undirected} is a consequence of~\eqref{eqn : MAIN/LowerTail/Directed}.
\end{proof}
\begin{proof}[Proof of Lemma~\ref{lem : LowerTail/WinCon}]
	Let $\beta \in \intervalleoo01$, $\zeta < \LimMassG(\beta )$ and $\eps \dpe \frac14\p{\LimMassG(\beta) - \zeta}$. By continuity of $\LimMassG$, there exists $\delta_1 \in \intervalleoo0{1- \beta}$ such that for all $\delta \in \intervalleof0{ \delta_1}$,
	\begin{equation}
		\label{eqn : LowerTail/WinCon/estime1}
		\LimMassG\p{ \frac{\beta}{1-\delta} } \ge \LimMassG(\beta) - \eps.
	\end{equation}
	By~\eqref{eqn : Preli/LLN/GDC}, there exists $L>0$ such that
	\begin{equation}
		\label{eqn : LowerTail/WinCon/estime2}
		\frac1L \E{ \MassGDC{1/2}{0}{ \frac{L \beta \base 1}{1-\delta} }{L} } \ge \LimMassG\p{ \frac{\beta}{1-\delta} } -\eps.
	\end{equation}
	Moreover, for small enough $\delta>0$, for large enough $\ell > 0$,
	\begin{equation}
		\label{eqn : LowerTail/WinCon/estime3}
		\frac1\ell \floor{\frac{(1-\delta)\ell }{L}} \E{ \MassGDC{1/2}{0}{ \frac{L \beta \base 1}{1-\delta} }{L} } \ge \frac1L \E{ \MassGDC{1/2}{0}{ \frac{L \beta \base 1}{1-\delta} }{L} } - \eps.
	\end{equation}
	Combining~\eqref{eqn : LowerTail/WinCon/estime1},~\eqref{eqn : LowerTail/WinCon/estime2} and~\eqref{eqn : LowerTail/WinCon/estime3} gives~\eqref{eqn : LowerTail/WinCon/Esperance}.

	Fix such values of $\delta$, $L$ and $\ell$. Let $k \in \intint{0}{\floor{ \frac{\delta \ell}{4d} }}^{d-1}\times \acc0$. The inclusion~\eqref{eqn : LowerTail/WinCon/Inclusion} is a consequence of
	\begin{equation}
		\label{eqn : LowerTail/WinCon/Inclusion/MainInequality}
		\MassGDF{0}{\ell \beta \base 1}{\ell} \ge  \sum_{n=0}^{\floor{\frac{(1-\delta)\ell }{L}} -1}\MassGDC{1/2}{2L k + \frac{ n L \beta \base 1}{1-\delta} }{ 2L k + \frac{ (n+1)L\beta \base 1}{1-\delta} }{L},
	\end{equation}
	which is a special case of Lemma~\ref{lem : Concatenation}.
\end{proof}


\section{Concentration inequalities}
\label{sec : Concentration}

In this section we prove Theorem~\ref{thm : BOUNDS}. In Section~\ref{sec : Concentration/General} we establish two concentration inequalities for functions of a Poisson point process which satisfy either the \emph{self-bounding} or the \emph{bounded differences} property, namely  Propositions~\ref{prop : Concentration/General/SB} and~\ref{prop : Concentration/General/BoundedDiff}. They are consequences of already known analogous inequalities for functions of a vector of independent variables, stemming from the so-called \emph{entropy method}. Lemma~\ref{lem : Concentration/General/Discretization} contains our discretization argument. Sections~\ref{subsec : Concentration/Greedy_is_SB} and~\ref{subsec : Concentration/Greedy_is_BD} are devoted to the special case where the function of interest is the mass of greedy animals or paths. Section~\ref{subsec : Concentration/Limit_Bound} concludes the proof of Theorem~\ref{thm : BOUNDS} by taking the limit in the aforementioned bounds.

\subsection{General bounds}
\label{sec : Concentration/General}

\begin{Definition}
	\label{def : Concentration/SB}
	Let $\progen$ be a Poisson process on a locally compact, second countable, Hausdorff topological space $\bbX$, with finite intensity; in particular $\progen$ has values in $\ProSpaceFinite[\bbX]$. Let $f : \ProSpaceFinite[\bbX] \rightarrow \intervallefo0\infty$ be a measurable function. 
	\begin{enumerate}[(i)]\item We say that $f$ is \emph{self-bounding} if for all $\eta \in \ProSpaceFinite[\bbX]$, for all atoms $x$ of $\eta$,
	\begin{equation}
		\label{eqn : Concentration/SB_1}
		0 \le f(\eta) - f(\eta - \Dirac x) \le 1,
	\end{equation}
	and
	\begin{equation}
		\label{eqn : Concentration/SB_2}
		\int_\bbX \p{f(\eta) - f(\eta - \Dirac x) } \eta(\d x) \le f(\eta).
	\end{equation}
	\item We say that $f$ satisfies the \emph{bounded differences} property with the constant $v>0$ if for all $\eta \in \ProSpaceFinite[\bbX]$,
	\begin{equation}
		\int_\bbX \p{f(\eta) - f(\eta - \Dirac x)}^2 \eta(\d x) \le v.
	\end{equation}
\end{enumerate}
\end{Definition}


\begin{Proposition}
	\label{prop : Concentration/General/SB}
	Let $\progen$ and $\bbX$ be as in Definition~\ref{def : Concentration/SB}. Let $f : \ProSpaceFinite[\bbX] \rightarrow \intervallefo0\infty$ be a self-bounding function. Then for all $t>0$,
	\begin{equation}
		\label{eqn : Concentration/General/SB}
		\Pb{f(\progen) \ge \E{f(\progen) } + t } \le \exp\p{ - h\p{\frac{t}{\E{f(\progen) }}}\E{f(\progen) } },
	\end{equation}
	where $h(s) = (1+s)\log(1+s) -s$.
\end{Proposition}


\begin{Proposition}
	\label{prop : Concentration/General/BoundedDiff}
	Let $\progen$ and $\bbX$ be as in Definition~\ref{def : Concentration/SB}. Let $f : \ProSpaceFinite[\bbX] \rightarrow \intervallefo0\infty$ be a mesurable, nondecreasing function. Assume that $f$ satisfies the bounded differences property with the constant $v>0$. Then for all $t>0$,
	\begin{equation}
		\label{eqn : Concentration/General/BoundedDiff}
		\Pb{f(\progen) \ge \E{f(\progen) } + t } \le \exp\p{ - \frac{t^2}{v} }.
	\end{equation}
\end{Proposition}

\begin{Lemma}
	\label{lem : Concentration/General/Discretization}
	Let $\progen$ and $\bbX$ be as in Definition~\ref{def : Concentration/SB}. Then there exists a family of point processes $\p{\progen_n^{(i)}}_{\substack{n\ge 1 \\ 1\le i \le n} }$ in $\DiracSpace[\bbX]$, with the following properties:
	\begin{enumerate}[(i)]
		\item For all $n\ge 1$, the variables $\p{ \progen_n^{(i)} }_{1\le i \le n}$ are independent.
		\item For all $n \ge 1$, almost surely,
			\begin{equation}
				\label{eqn : Concentration/General/Discretization/UB}
				\progen_n \dpe \sum_{i=1}^{n} \progen_n^{(i)} \le \progen.
			\end{equation}
		\item Almost surely, for large enough $n$,~\eqref{eqn : Concentration/General/Discretization/UB} is an equality.
	\end{enumerate}
\end{Lemma}
\begin{proof}
	Let $\rho$ be the intensity measure of $\progen$ and $n \in \N^*$. Let $\progen^*$ be a Poisson point process on $\bbX \times \intervallefo01T$, with intensity measure $\rho \otimes \Leb[1]$. For all $i\in \intint1{n}$, we define the measure $\progen_{n}^{(i)}$ on $\bbX$ by
	\begin{equation*}
		\progen_n^{(i)}(B) \dpe \progen^*\p{B \times \intervallefo{\frac{i-1}{n} }{\frac{i}{n} } }, \text{ for all Borel subset }B\subseteq \R^d.
	\end{equation*}
	By standard arguments (see e.g.\ Theorems 5.1 and 5.2 in \cite{LastPenrose}) the $\progen_n^{(i)}$, for $i\in\intint1n$, are independent and they are Poisson point processes with intensity measure $\rho/n$. Moreover, by the so-called \emph{superposition principle} (see e.g.\ Theorem 3.3 in \cite{LastPenrose}), their sum has the same distribution as $\progen$. By~\eqref{eqn : Intro/Sum_of_Diracs_General}, there exists a measurable map $\Phi : \ProSpaceFinite[\bbX] \rightarrow \DiracSpace[\bbX]$, such that for all $\eta \in \ProSpaceFinite[\bbX]$:
	\begin{enumerate}[(i)]
		\item If $\eta = 0$ then $\Phi(\eta) = 0$.
		\item If $\eta \neq 0$, then $\Phi(\eta) = \Dirac x$, where $x$ is an atom of $\eta$.
	\end{enumerate}
	Almost surely, for all pair of atoms $(x,t)$ and $(x', t')$ of $\progen^*$, $t\neq t'$, therefore almost surely, for large enough $n$, $\progen_n = \progen$. 
\end{proof}

\begin{proof}[Proof of Proposition~\ref{prop : Concentration/General/SB}]
	First note that $0 \le f(\progen) \le \progen(\bbX)$, therefore $f(\progen)$ is integrable. Let $n\ge 1$. Consider the function
	\begin{align}
		f_n : \DiracSpace[\bbX]^{n} &\longrightarrow \intervallefo0\infty \eol
			\p{\eta_i}_{1\le i \le n} &\longmapsto f\p{ \sum_{i=1}^{n} \eta_i}. \label{eqn : Concentration/General/SB/Def_fn}
	\end{align}
	Then $f_n$ satisfies the self-bounding property in the discrete sense, as defined in \cite{BLM}, Section 3.3. Consequently, by Theorem 6.12 in \cite{BLM}, for all $t>0$,
	\begin{equation*}
		\begin{split}
		&\Pb{ f_n\p{\progen_n^{(1)}, \dots, \progen_n^{(n)} } \ge \E{f_n\p{\progen_n^{(1)}, \dots, \progen_n^{(n)} }  } + t }%
			\\ &\quad \le \exp\p{ -h\p{\frac{t}{\E{f_n\p{\progen_n^{(1)}, \dots, \progen_n^{(n)} }}} }\E{f_n\p{\progen_n^{(1)}, \dots, \progen_n^{(n)} }  } }.
		\end{split}
	\end{equation*}
	In other words,
	\begin{equation}
		\Pb{ f\p{\progen_n} \ge \E{f\p{\progen_n}  } + t }%
			\le \exp\p{ -h\p{\frac{t}{\E{f\p{\progen_n}}} }\E{f\p{\progen_n}  } }.
	\end{equation} 
	Moreover, the last two items in Lemma~\ref{lem : Concentration/General/Discretization} yield the convergence
	\begin{equation}
		f(\progen_n) \xrightarrow[n\to\infty]{\text{a.s.\ and }\rL^1} f(\progen),
	\end{equation}
	thus~\eqref{eqn : Concentration/General/SB}.
\end{proof}
\begin{proof}[Proof of Proposition~\ref{prop : Concentration/General/BoundedDiff}]
	First note that $0 \le f(\progen) \le \sqrt v \progen(\bbX)$, therefore $f(\progen)$ is integrable. The rest of the proof is like Proposition~\ref{prop : Concentration/General/SB}, based on Theorem 6.7 in \cite{BLM} instead of Theorem 6.12 there.
\end{proof}

\subsection{Concentration bounds for unpenalized greedy animals and paths}
\label{subsec : Concentration/Greedy_is_SB}
The main result of this section, Proposition~\ref{prop : Concentration/Unpen}, gives an upper-tail concentration bound for the mass of the (unpenalized) mass of greedy paths and animals. The letter $\mathrm G$ denotes either $\mathrm A$ or $\mathrm P$.
\begin{Proposition}
	\label{prop : Concentration/Unpen}
	Assume that $\nu$ is supported by $\intervalleof0m$, with $0<m<\infty$. Let $\ell, t \in \intervalleoo0\infty$. Then
	\begin{equation}
		\label{eqn : CONCENTRATION/Undirected_Unpen}
		\Pb{\MassGUF \ell \ge \E{\MassGUF\ell} + t } \le \exp\p{ -h\p{\frac{t}{\E{\MassGUF\ell} }} \frac{\E{\MassGUF\ell} }{m}  },
	\end{equation}
	where $h(s) = (1+s)\log(1+s) - s$. Moreover, for all $\beta \in \intervallefo01$,
	\begin{equation}
		\label{eqn : CONCENTRATION/Directed_Unpen}
		\begin{split}
		&\Pb{\MassGDF0{\ell \beta \base 1}\ell \ge \E{\MassGDF0{\ell \beta \base 1}\ell} + t } \\ &\quad\le \exp\p{ -h\p{\frac{t}{\E{\MassGDF0{\ell \beta \base 1}\ell} }} \frac{\E{\MassGDF0{\ell \beta \base 1}\ell} }{m}  }.
		\end{split}
	\end{equation}
\end{Proposition}
\begin{proof}
	We only show the first part since the second one is similar. By a standard truncation argument, we may assume that $\nu\p{\intervalleof0m} < \infty$.	Moreover, by Proposition~\ref{prop : Concentration/General/SB}, it is sufficient to check that \begin{align*}
		\ProSpaceFinite[\R^d \times \intervalleof0m] &\longrightarrow \intervallefo0\infty \\
		\eta &\longmapsto \frac1m\MassGUF \ell[\eta]
	\end{align*}
	is self-bounding. Let $\eta \in \ProSpaceFinite[\R^d \times \intervalleof0m]$. Since $\eta$ is finite, there exists $\xi \in \SetGUF \ell$ such that 
	\begin{equation*}
		\MassGUF \ell [\eta] = \Mass\xi[\eta].
	\end{equation*}
	Let $x=(z,s)\in \R^d \times \intervalleof0m$ be an atom of $\eta$. It is clear that
	\begin{equation}
		\label{eqn : Concentration/Greedy_is_SB/Unpen/bound1}
		\MassGUF \ell [\eta - \Dirac x] \le \MassGUF\ell [\eta].
	\end{equation}
	Moreover,
	\begin{equation*}
		\Mass{\xi}[\eta - \Dirac x] = \Mass{\xi}[\eta] - s \ind{z \in \xi} =  \MassGUF \ell [\eta] -s\ind{z \in \xi},
	\end{equation*}
	thus
	\begin{align}
		\MassGUF \ell [\eta]- \MassGUF \ell [\eta-\Dirac x] &\le s \ind{z \in \xi}, \nonumber
		\intertext{therefore}
		\label{eqn : Concentration/Greedy_is_SB/Unpen/bound2}
		\int_{\R^d \times \intervalleof0m} \p{ \MassGUF \ell [\eta]- \MassGUF \ell [\eta-\Dirac x] }\eta(\d x) &\le \Mass{\xi}[\eta] = \MassGUF \ell [\eta],
	\end{align}
	which concludes the proof.
\end{proof}

\subsection{Concentration bound for penalized animals}
\label{subsec : Concentration/Greedy_is_BD}
The main result of this section, Proposition~\ref{prop : Concentration/Pen}, gives an upper-tail concentration bound for the mass the penalized mass of greedy animals. We fix $0<q<\infty$ and assume that $\nu$ is finite. Otherwise, the question is already answered by Proposition~\ref{prop : Concentration/Unpen}, as for all $\ell > 0$ and $u \in \ball{0,1}$, the variables $\MassAUF\ell$ and $\MassADF0{\ell u}\ell$ are a.s.\ equal to their penalized counterparts.

Let $\proDirac$ denote the image of $\pro$ through the map
\begin{align*}
	\R^d \times \intervalleoo0\infty &\longrightarrow \R^d \times \intervalleoo0\infty \\
	(z, s) &\longmapsto (z,1).
\end{align*}
By the mapping theorem (see e.g.\ Theorem 5.1 in \cite{LastPenrose}), $\proDirac$ is a Poisson point process with intensity $\nu\p{\intervalleoo0\infty} \Leb\otimes \Dirac 1$. For all $\ell>0$, consider
\begin{equation}
	\MassN\ell \dpe \MassAUF \ell \cro{\proDirac }.
\end{equation}

\begin{Proposition}
	\label{prop : Concentration/Pen}
	Assume that $\nu$ is supported by $\intervalleof0m$, with $0<m<\infty$, and is finite. Let $\ell, t , \alpha \in \intervalleoo0\infty$.
	Then
	\begin{equation}
		\label{eqn : CONCENTRATION/Undirected_Pen}
		\Pb{\MassAUFpen \ell q\ge \E{\MassAUFpen\ell q} + t } \le \exp\p{ - h\p{ \alpha } \E{ \MassN\ell } } +  \exp\p{ \frac{- t^2 }{ (1+\alpha)\E{ \MassN\ell } (q+m)^2 } },
	\end{equation}
	where $h(s) = (1+s)\log(1+s) - s$. Moreover, for all $u\in \clball{0,1}$,
	\begin{equation}
		\label{eqn : CONCENTRATION/Directed_Pen}
		\begin{split}
		&\Pb{\MassADFpen {0}{\ell u}\ell q \ge \E{\MassADFpen {0}{\ell u}\ell q} + t }  \\ &\quad \le \exp\p{ - h\p{ \alpha } \E{ \MassN\ell } } +  \exp\p{ \frac{- t^2 }{ (1+\alpha)\E{ \MassN\ell } (q+m)^2 } }.
		\end{split}
	\end{equation}
\end{Proposition}

\begin{proof}[Proof of Proposition~\ref{prop : Concentration/Pen}]
	Let us fix $\ell, t , \alpha \in \intervalleoo0\infty$. We only prove~\eqref{eqn : CONCENTRATION/Undirected_Pen}, the other part being similar. We define
	\begin{equation}
		\hatMassAUFpen\ell q \dpe \max \set{  \Masspen A  - q\#\p{\xi \setminus A}  }{ \begin{array}{c} \xi \in \SetAUFalt\ell \\ A\subseteq \xi \\ \# A \le (1+\alpha) \E{\MassN \ell}   \end{array} }.
	\end{equation}
	Proposition~\ref{prop : Concentration/Unpen} implies that
	\begin{equation}
		\label{eqn : Concentration/Pen/Dirac}
		\Pb{\MassN\ell  \ge (1+\alpha)\E{ \MassN\ell }  } \le \exp\p{ - h\p{\alpha } \E{ \MassN\ell } }.
	\end{equation}
	We claim that
	\begin{equation}
		\label{eqn : Concentration/Pen/2ndterm}
		\Pb{ \hatMassAUFpen\ell q \ge \E{\hatMassAUFpen\ell q} + t } \le \exp\p{ - \frac{- t^2 }{ (1+\alpha)\E{ \MassN\ell } (q+m)^2 } }.
	\end{equation}
	Indeed let us consider 
	\begin{align*}
		f : \ProSpaceFinite[\R^d \times \intervalleof0m ] &\longrightarrow \intervallefo0\infty\\
		\eta &\longmapsto  \hatMassAUFpen\ell q\cro\eta.
	\end{align*}
	By Proposition~\ref{prop : Concentration/General/BoundedDiff}, it is sufficient to show that for all $\eta \in \ProSpaceFinite[\R^d \times \intervalleof0m ]$,
	\begin{equation}
		\label{eqn : Concentration/Pen/Hat/Hypothese}
		\int_{\R^d \times \intervalleof0m} \p{f(\eta) - f(\eta - \Dirac{(x,s)})}^2 \eta(\d x, \d s) \le (1+\alpha) \E{ \MassN\ell } (q+m)^2 .
	\end{equation}
	Fix $\eta \in \ProSpaceFinite[\R^d \times \intervalleof0m ]$. There exists $\xi \in \SetAUFalt\ell$ and a set of points $A\subseteq \xi$ such that $\# A \le (1+\alpha) \E{\MassN \ell}$ and
	\begin{equation*}
		f(\eta) = \Masspen A\cro{\eta} - q\#\p{\xi \setminus A}. 
	\end{equation*}
	Let $(x,s)$ be an atom of $\eta$. If $x \notin A$, then \[\Masspen \xi\cro\eta = \Masspen \xi\cro{\eta - \Dirac{(x,s)} }.\]If $x\in A$, then \[ \Masspen \xi\cro\eta - \Masspen \xi\cro{\eta - \Dirac{(x,s)} } \le q+m. \] Consequently,
	\[ 0\le f(\eta) - f(\eta - \Dirac{(x,s)}) \le (q+m) \ind{x\in A}, \]
	thus~\eqref{eqn : Concentration/Pen/Hat/Hypothese}, therefore~\eqref{eqn : Concentration/Pen/2ndterm}.

	The inclusion
	\begin{equation*}
		\acc{\MassN\ell  < (1+\alpha)\E{ \MassN\ell }  } \subseteq \acc{ \hatMassAUFpen\ell q = \MassAUFpen\ell q }
	\end{equation*}
	implies that
	\begin{align*}
		\acc{\MassAUFpen \ell q\ge \E{\MassAUFpen\ell q} + t }
			&\quad \subseteq \acc{ \hatMassAUFpen\ell q \ge \E{\MassAUFpen\ell q} + t } \cup \acc{ \MassN\ell  \ge (1+\alpha)\E{ \MassN\ell } } \\
			&\quad \subseteq \acc{ \hatMassAUFpen\ell q \ge \E{\hatMassAUFpen\ell q} + t } \cup \acc{ \MassN\ell  \ge (1+\alpha)\E{ \MassN\ell } }.
	\end{align*}
	The union bound,~\eqref{eqn : Concentration/Pen/Dirac} and~\eqref{eqn : Concentration/Pen/2ndterm} concludes the proof.
\end{proof}

\subsection{Proof of Theorem~\ref{thm : BOUNDS}}
\label{subsec : Concentration/Limit_Bound}

Assume that $\nu$ is supported by $\intervalleof0m$, with $0<m<\infty$, and $\mathrm G = \mathrm P$ or $\mathrm A$. Let $\beta \in \intervallefo01$ and $\zeta > \LimMassG(\beta)$. Let $0 < \eps < \zeta  - \LimMassG(\beta)$. For large enough $\ell>0$,
	\[ \E{\MassGDF0{\ell \beta \base 1}\ell} \le \ell\p{\LimMassG(\beta) + \eps},  \]
	therefore, by~\eqref{eqn : CONCENTRATION/Directed_Unpen},
	\begin{align}
		\Pb{ \MassGDF0{\ell \beta \base 1}\ell \ge \zeta \ell }%
			&\le \Pb{ \vphantom\int \MassGDF0{\ell \beta \base 1}\ell \ge \E{\MassGDF0{\ell \beta \base 1}\ell} + \p{\zeta - \LimMassG(\beta) - \eps} \ell }\eol
			&\le \exp\p{ -h\p{\frac{\p{\zeta - \LimMassG(\beta) - \eps} \ell}{\E{\MassGDF0{\ell \beta \base 1}\ell} }} \frac{\E{\MassGDF0{\ell \beta \base 1}\ell} }{m}  }. \nonumber
	\end{align}
	Taking the $\log$, multiplying by $-\frac 1\ell$ and letting $\ell \to \infty$ yields
	\begin{equation*}
		\FdT_\rG(\beta,\zeta) \ge h\p{ \frac{\zeta - \LimMassG(\beta) - \eps}{ \LimMassG(\beta) }  } \frac{\LimMassG(\beta)}{m}.
	\end{equation*}
	Letting $\eps\to0$, we get~\eqref{eqn : BOUNDS/Unpen}.

We use similar arguments to obtain the bound~\eqref{eqn : BOUNDS/Pen} from~\eqref{eqn : CONCENTRATION/Directed_Pen}, provided
\begin{equation}
	\label{eqn : Concentration/Pen/Scaling}
	\lim_{n\to\infty} \frac{\E{\MassN \ell} }{\ell} = \Cr{DIRAC} \nu\p{\intervalleoo0\infty}^{1/d},
\end{equation}
for a constant $\Cl{DIRAC}>0$ depending only on the dimension. To see this, first note that by Theorem~\ref{thm : Intro/LLN}, the limit in~\eqref{eqn : Concentration/Pen/Scaling} exists and is finite. Moreover, by a scaling argument (see Lemma 2.2 in \cite{Article3}) it is proportional to $\nu\p{\intervalleoo0\infty}^{1/d}$. \qed


\section{Upper-tail large deviation estimate}
\label{sec : UpperTail}

We recall that $\rG$ denotes either $\rP$ or $\rA^{(q)}$, for some $q\in \intervalleff0\infty$. In this section we prove Theorem~\ref{thm : MAIN/UpperTail}. To do so, we show by a subadditivity argument the existence and finiteness of the analogue of \eqref{eqn : MAIN/UpperTail/Directed} for diamond animals defined~in Section~\ref{subsec : Preli/AuxProcesses} (see Proposition~\ref{prop : UpperTail/LimitDiamAntidiam}). We then extend this result to $\MassGDF\cdot\cdot\cdot$ (see Proposition~\ref{prop : UpperTail/Limit}), and finally prove that the upper-tail deviations are indeed of order $\exp\p{-\ell}$ (see Proposition~\ref{prop : UpperTail/OdG}). 
\begin{Proposition}
	\label{prop : UpperTail/LimitDiamAntidiam}
	There exists a nondecreasing, convex, continuous fonction $\FdT :\intervallefo01 \times \intervalleoo0\infty \rightarrow \intervallefo0\infty$ such that for all $\delta\in \intervalleof01$, $\beta \in \intervalleoo01$ and $\zeta \in \intervalleoo0\infty$,
	\begin{align}
		\label{eqn : UpperTail/LimitDiamAntidiam/Diam}
		\FdT(\beta, \zeta) &= \lim_{\ell \to \infty} -\frac 1\ell \log\Pb{ \MassGDC\delta{0}{\ell \beta \base 1}{\ell} \ge \ell \zeta } = \inf_{\ell \ge1} -\frac 1\ell \log\Pb{ \MassGDC\delta{0}{\ell \beta \base 1}{\ell} \ge \ell \zeta }.
	\end{align}
	Moreover, $\FdT$ satisfies Theorem~\ref{thm : MAIN/UpperTail}\eqref{item : MAIN/UpperTail/ConvexNondec}.
\end{Proposition}
\begin{Proposition}
	\label{prop : UpperTail/Limit}
	For all $\beta \in \intervallefo01$ and $\zeta \in \intervalleoo0\infty$,
	\begin{equation}
		\label{eqn : UpperTail/Limit}
		\FdT(\beta,\zeta) = \lim_{\ell \to \infty} -\frac1\ell \log\Pb{ \MassGDF{0}{\ell \beta \base 1}{\ell} \ge \ell \zeta }.
	\end{equation}
	For all $\zeta\in \intervalleoo0\infty$,
	\begin{equation}
		\label{eqn : UpperTail/LimitFree}
		\FdT(0,\zeta) = \lim_{\ell \to \infty} -\frac1\ell \log\Pb{ \MassGUF\ell \ge \ell \zeta }.
	\end{equation}
\end{Proposition}
\begin{Proposition}
	\label{prop : UpperTail/OdG}
	Assume that~\eqref{ass : ExpMoment} holds. For all $\beta \in \intervallefo01$ and $\zeta \in \intervalleoo0\infty$, if then $\zeta > \LimMassG(\beta)$ then $\FdT(\beta,\zeta) >0$.
\end{Proposition}
Sections~\ref{subsec : UpperTail/DiamAntidiam}, \ref{subsec : UpperTail/Limit} and \ref{subsec : UpperTail/OdG} are devoted to the proofs of Propositions~\ref{prop : UpperTail/LimitDiamAntidiam}, \ref{prop : UpperTail/Limit} and \ref{prop : UpperTail/OdG} respectively. Let us check that they indeed imply Theorem~\ref{thm : MAIN/UpperTail}.

\begin{proof}[Proof of Theorem~\ref{thm : MAIN/UpperTail}]
	Everything is clear except the converse sense in~\eqref{item : MAIN/UpperTail/OdG}, and~\eqref{item : MAIN/UpperTail/Limit}. Let $\beta \in \intervallefo01$. By~\eqref{eqn : Intro/LLN/Directed}, \[ \forall \zeta < \LimMassG(\beta), \quad \FdT(\beta, \zeta) = 0. \]
	Moreover, since $\FdT$ is continuous, this still holds for $\zeta = \LimMassG(\beta)$, thus~\eqref{item : MAIN/UpperTail/OdG}.

	Let $\LimMassG(\beta) < \zeta_1 < \zeta_2$. By convexity,
	\[
	\frac{\FdT(\beta, \zeta_2) - \FdT(\beta, \zeta_1)}{\zeta_2 - \zeta_1} \ge \frac{\FdT(\beta, \zeta_1) - \FdT(\beta, \LimMassG(\beta))}{\zeta_1 - \LimMassG(\beta) } >0,
	\]
	therefore $\zeta \mapsto \FdT(\beta, \zeta)$ is strictly increasing on $\intervallefo{\LimMassG(\beta)}\infty$, and diverges to $\infty$ as $\zeta \to \infty$.
\end{proof}

\subsection{Existence of the rate function for diamond animals}
\label{subsec : UpperTail/DiamAntidiam}

In this section we prove Proposition~\ref{prop : UpperTail/LimitDiamAntidiam}, which is a consequence of Lemmas~\ref{lem : UpperTail/LimitDiamAntidiam/Diam},~\ref{lem : UpperTail/LimitDiamAntidiam/Antidiam} and~\ref{lem : UpperTail/LimitDiamAntidiam/Convex}. We essentially follow the proof of Proposition~2.1 in \cite{Article2}, simplifying some arguments thanks to rotation-invariance. Our concatenation arguments rely on Lemma~\ref{lem : Concatenation}.
\begin{Lemma}
	\label{lem : UpperTail/LimitDiamAntidiam/Diam}
	For all $\delta \in \intervalleof01$, $\beta \in \intervalleoo01$ and $\zeta \in \intervalleoo0\infty$, the limit
	\begin{equation}
		\FdTDiamond(\beta, \zeta) \dpe \lim_{\ell \to \infty} -\frac 1\ell \log\Pb{ \MassGDC\delta{0}{\ell \beta \base 1}{\ell} \ge \ell \zeta }
	\end{equation}
	exists and is finite.
\end{Lemma}
\begin{Lemma}
	\label{lem : UpperTail/LimitDiamAntidiam/Antidiam}
	For all $\beta \in \intervalleoo01$ and $\zeta \in \intervalleoo0\infty$, $\delta \mapsto \FdTDiamond(\beta, \zeta)$ is constant on $\intervalleof01$. We will simply write it $\FdT(\beta, \zeta)$.
\end{Lemma}
\begin{Lemma}
	\label{lem : UpperTail/LimitDiamAntidiam/Convex}
	The function $\FdT$ has a nondecreasing, convex, continuous extension on $\intervallefo01\times \intervalleoo0\infty$.
\end{Lemma}
\begin{proof}[Proof of Lemma~\ref{lem : UpperTail/LimitDiamAntidiam/Diam}]
	Let $\delta$, $\beta$, $\zeta$ be as in the lemma. Let $\ell_1, \ell_2 \ge 0$. Since $\SetGDC\delta{0}{\ell_1 \beta \base 1}{\ell_1}$ and $\SetGDC\delta{\ell_1}{ (\ell_1 + \ell_2) u\beta \base 1}{\ell_2}$ concatenate into $\SetGDC\delta{0}{(\ell_1 + \ell_2) \beta \base 1}{\ell_1 + \ell_2}$, by Lemma~\ref{lem : Concatenation},
	\begin{equation}
		\MassGDC\delta{0}{\ell_1 \beta \base 1}{\ell_1} + \MassGDC\delta{\ell_1 \beta \base1}{ (\ell_1 + \ell_2) \beta \base 1}{\ell_2}%
			\le \MassGDC\delta{0}{(\ell_1 + \ell_2) \beta \base 1}{\ell_1 + \ell_2},\nonumber \end{equation}
	and the terms of the sum are independent. Consequently,
	\begin{align}
	\begin{split}
		&\Pb{ \MassGDC\delta{0}{\ell_1 u}{\ell_1} \ge \ell_1 \zeta} \cdot \Pb{\MassGDC\delta{\ell_1}{ (\ell_1 + \ell_2) \beta \base 1}{\ell_2} \ge \ell_2 \zeta}\\ %
			& \quad \le \Pb{ \MassGDC\delta{0}{(\ell_1 + \ell_2) \beta \base 1}{\ell_1 + \ell_2} \ge (\ell_1 + \ell_2)\zeta}, \end{split} \nonumber
	\intertext{thus by stationarity,}
	\label{eqn : UpperTail/LimitDiamAntidiam/Diam/Subadditivity}
		\begin{split}
		&-\log \Pb{ \MassGDC\delta{0}{\ell_1 \beta \base 1}{\ell_1} \ge \ell_1 \zeta} -\log \Pb{\MassGDC\delta{0}{ \ell_2 \beta \base 1}{\ell_2} \ge \ell_2 \zeta} \\%
			&\quad \ge -\log \Pb{ \MassGDC\delta{0}{(\ell_1 + \ell_2) \beta \base 1}{\ell_1 + \ell_2} \ge (\ell_1 + \ell_2)\zeta}.
		\end{split}
	\end{align}
	Besides, for all $\ell \in \intervalleff12$,
	\begin{equation*}
		\acc{ \MassGDC{\delta}0{ \ell \beta \base 1}{\ell} \ge \zeta \ell  } \supseteq \acc{ \MassGDC\delta0{\beta \base 1}1 \ge 2\zeta },
	\end{equation*}
	hence
	\begin{equation}
		\label{eqn : UpperTail/LimitDiamAntidiam/Diam/Finiteness}
		\sup_{1\le \ell \le 2} -\log \Pb{ \MassGDC\delta{0}{\ell \beta \base 1}{\ell} \ge \ell \zeta } < \infty.
	\end{equation}
	The lemma follows from~\eqref{eqn : UpperTail/LimitDiamAntidiam/Diam/Subadditivity},~\eqref{eqn : UpperTail/LimitDiamAntidiam/Diam/Finiteness} and Fekete's lemma.
\end{proof}
\begin{proof}[Proof of Lemma~\ref{lem : UpperTail/LimitDiamAntidiam/Antidiam}]
	Let $\beta$, $\zeta$ be as in the lemma, and $\delta\in \intervalleof01$. For all $\ell>0$,
	\begin{equation*}
		\acc{\MassGDC{\delta}{0}{\ell \beta \base 1}{\ell} \ge \zeta \ell } \subseteq \acc{\MassGDC{1}{0}{\ell \beta \base 1}{\ell} \ge \zeta \ell },
	\end{equation*}
	therefore by Lemma~\ref{lem : UpperTail/LimitDiamAntidiam/Diam},
	\begin{equation}
		\label{eqn : UpperTail/LimitDiamAntidiam/Antidiam/UB}
		\FdTDiamond[\delta](\beta \base 1, \zeta) \le \FdTDiamond[1](\beta \base 1,\zeta).
	\end{equation}

	Let us show the converse inequality. It follows from the proof of Lemma 2.3 in \cite{Article2} that there exists a constant $\Cl{AD} = \Cr{AD}(\delta)>0$ such that for all $\ell>0$, the sets $\SetGDC{\delta}{-\Cr{AD} \beta \base 1\ell }{0}{\Cr{AD}\ell}$, $\SetGDC{1}{0}{\ell  \beta \base 1}{\ell}$ and $\SetGDC{\delta}{\ell  \beta \base 1 }{ (1+ \Cr{AD})\ell \beta \base1 }{\Cr{AD}\ell}$ concatenate into $\SetGDC{\delta}{-\Cr{AD} \beta \base 1\ell }{(1+ \Cr{AD})\ell  \beta \base 1}{ (1+2\Cr{AD})\ell}$. Consequently, by Lemma~\ref{lem : Concatenation},
	\begin{equation*}
		\begin{split}
		&\Pb{ \MassGDC{\delta}{-\Cr{AD}\beta \base 1\ell }{0}{\Cr{AD}\ell} \ge \Cr{AD}\ell \zeta } \cdot \Pb{ \MassGDC{1}{0}{\ell \beta \base 1}{\ell} \ge \ell \zeta } \cdot \Pb{ \MassGDC{\delta}{\ell \beta \base 1 }{(1+ \Cr{AD})\ell \beta \base 1}{\Cr{AD}\ell} \ge \Cr{AD}\ell \zeta } \\
		&\quad \le \Pb{ \MassGDC{\delta}{-\Cr{AD}\beta \base 1\ell }{(1+ \Cr{AD})\ell \beta \base 1}{ (1+2\Cr{AD})\ell} \ge (1+2\Cr{AD})\ell \zeta }.
		\end{split}
	\end{equation*}
	By Lemma~\ref{lem : UpperTail/LimitDiamAntidiam/Diam} and stationarity of the model, taking the $\log$, multiplying by $-\frac1\ell$ and letting $\ell \to \infty$ gives
	\begin{equation}
		\label{eqn : UpperTail/LimitDiamAntidiam/Antidiam/LB}
		\FdTDiamond[1](\beta, \zeta) \ge \FdTDiamond(\beta,\zeta),
	\end{equation}
	which concludes the proof.
\end{proof}

\begin{proof}[Proof of Lemma~\ref{lem : UpperTail/LimitDiamAntidiam/Convex}]
	We first show that $\FdT$ is convex on $\intervalleoo01 \times \intervalleoo0\infty$. Let $\beta_1, \beta_2 \in \intervalleoo01$, and $\theta_1, \theta_2 \in \intervalleff01$ be such that $\theta_1+\theta_2=1$. Then $\SetGDC{1/2}{0}{\theta_1\ell \beta_1 \base 1}{ \theta_1 \ell}$ and $\SetGDC{1}{\theta_1 \ell \beta_1 \base 1}{ \theta_1 \ell \beta_1 \base 1+ \theta_2 \ell \beta_2 \base 1  }{\theta_2 \ell}$ concatenate into $\SetGDC{1}{0}{ \ell( \theta_1 \beta_1  + \theta_2 \beta_2 )\base 1 }{\ell}$. By Lemma~\ref{lem : Concatenation}, for all $\ell>0$ and $\zeta_1, \zeta_2 >0$,
	\begin{equation*}
		\begin{split}
		&\Pb{\MassGDC{1}{0}{\theta_1\ell \beta_1 \base1 }{ \theta_1 \ell} \ge \theta_1 \ell  \zeta_1} \cdot \Pb{ \MassGDC{1}{\theta_1 \ell \beta_1 \base1 }{ \theta_1 \ell \beta_1 \base1 + \theta_2 \ell \beta_2 \base 1   }{\theta_2 \ell} \ge \theta_2 \ell \zeta_2 } \\
		&\quad \le \Pb{ \MassGDC{1}{0}{ \ell( \theta_1 \beta_1 \base1  + \theta_2 \beta_2 \base 1 ) }{\ell} \ge (\theta_1 + \theta_2)\ell }.
		\end{split}
	\end{equation*}
	By Lemmas~\ref{lem : UpperTail/LimitDiamAntidiam/Diam} and~\ref{lem : UpperTail/LimitDiamAntidiam/Antidiam}, taking the $\log$, multiplying by $-\frac1\ell$ and letting $\ell \to \infty$ gives
	\begin{equation}
		\label{eqn : UpperTail/LimitDiamAntidiam/WeakConvexity}
		\FdT( \theta_1 \beta_1    + \theta_2 \beta_2   , \theta_1 \zeta_1 + \theta_2 \zeta_2) \le \theta_1 \FdT(\beta_1   , \zeta_1) + \theta_2 \FdT(\beta_2   , \zeta_2),
	\end{equation}
	i.e.\ the announced convexity. In particular it is continuous on $\intervalleoo01 \times \intervalleoo0\infty$.

	The function $\FdT$ is clearly nondecreasing with respect to its second coordinate. Let us show that it is also nondecreasing with respect to the first one. Let $0< \beta_1 \le \beta_2 <1$ and $\zeta > 0$. Let $\delta>0$ be small enough so that
	\begin{equation}
		\Diamant \delta{0}{\beta_1 \base 1 + \sqrt{\beta_2^2 - \beta_1^2} \base 2 } \cap \Diamant \delta{\beta_1 \base 1 + \sqrt{\beta_2^2 - \beta_1^2} \base 2 }{2\beta_1 \base 1} = \emptyset
	\end{equation}
	and
	\begin{equation}
		\Diamant\delta{0}{\beta_1 \base 1 + \sqrt{\beta_2^2 - \beta_1^2} \base 2 }, \Diamant\delta{\beta_1 \base 1 + \sqrt{\beta_2^2 - \beta_1^2} \base 2 }{2\beta_1} \subseteq \Gap{0}{2\beta_1 \base 1}.
	\end{equation}
	This implies that for all $\ell>0$, the sets \[ \SetGDC\delta{0}{\ell\p{\beta_1 \base 1 + \sqrt{\beta_2^2 - \beta_1^2} \base 2 } }\ell \text{ and } \SetGDC\delta{ \ell \p{\beta_1 \base 1 + \sqrt{\beta_2^2 - \beta_1^2} \base 2 } }{2\ell \beta_1 \base 1}\ell \] are concatenable in $\SetGDC 1{0}{2\ell \beta_1 \base 1}{2\ell}$. Consequently, by Lemma~\ref{lem : Concatenation}, for all $\ell >0$,
	\begin{equation*}
		\begin{split}
			\Pb{ \MassGDC\delta{0}{2\beta_1 \base 1}{2\ell} \ge 2\zeta \ell } &\ge \Pb{ \MassGDC\delta{0}{\ell\p{\beta_1 \base 1 + \sqrt{\beta_2^2 - \beta_1^2} \base 2 } }\ell \ge \zeta \ell} \\
			&\qquad \cdot \Pb{  \MassGDC\delta{\ell\p{\beta_1 \base 1 + \sqrt{\beta_2^2 - \beta_1^2} \base 2 } }{2\ell \beta_1 \base 1}\ell \ge \zeta  \ell}.
		\end{split}
	\end{equation*}
	Since the model is stationary and rotation-invariant,
	\begin{equation*}
		\Pb{ \MassGDC\delta{0}{2\beta_1 \base 1}{2\ell} \ge 2\zeta \ell } \ge \Pb{ \MassGDC\delta{0}{ \ell \beta_2 \ell }\ell \ge \zeta \ell}^2.
	\end{equation*}
	Taking the $\log$, multiplying by $-1/\ell$ and letting $\ell \to \infty$ leads to
	\begin{equation}
		\FdT(\beta_1, \zeta) \le \FdT(\beta_2, \zeta).
	\end{equation}

	We extend $\FdT$ to $\intervallefo01 \times \intervalleoo0\infty$ by defining
	\begin{equation}
		\FdT(0, \zeta) \dpe \lim_{\beta \to 0} \FdT(\beta, \zeta),
	\end{equation}
	for $\zeta>0$. It is still convex. In particular, $\FdT(0, \cdot)$ is continuous on $\intervalleoo0\infty$. Furthermore, the family of continuous functions $\p{\FdT(\beta, \cdot)}_{0< \beta < 1}$ converges monotonously to $\FdT(0, \cdot)$ as $b\to0$, therefore by Dini's theorem, $\FdT(\beta, \cdot)$ converges uniformly to $\FdT(0,\cdot)$ on compact subsets of $\intervalleoo0\infty$. Consequently, $\FdT$ is continuous on $\intervalleoo01 \times \intervalleoo0\infty$.
\end{proof}

\subsection{Existence of the rate function for unconstrained animals}
\label{subsec : UpperTail/Limit}
In this section we prove Proposition~\ref{prop : UpperTail/Limit}. Since the proof of~\eqref{eqn : UpperTail/Limit} and~\eqref{eqn : UpperTail/LimitFree} are nearly identical, we only treat the former.\footnote{Notice that in the case $\rG= \rA = \rA^{(0)}$,~\eqref{eqn : UpperTail/LimitFree} is actually a special case of~\eqref{eqn : UpperTail/Limit}. }  In the case $\mathrm G = \mathrm A^{(q)}$, for some $q\in \intervalleff0\infty$, it relies on the elementary observation that any animal is a cylinder animal, thus we can conclude with the union bound and Proposition~\ref{prop : UpperTail/LimitDiamAntidiam}. In the case $\rG = \rP$, that observation fails, but an extra step consisting in writing any path as a concatenation of cylinder paths fixes the issue. Existence of such a decomposition is ensured by Lemma~\ref{lem : UpperTail/Limit/Paths/Decomposition}, which is an easier version of Lemma 2.11 in \cite{Article2} (see also Lemma 7.7 in \cite{Mar02}). In both cases we will need Lemma~\ref{lem : UpperTail/Limit/Bound}, proven at the end of the section.

\begin{Lemma}
	\label{lem : UpperTail/Limit/Paths/Decomposition}
	Let $0<\delta <1$, $\beta \in \intervallefo01$, $\ell>0$ and $\gamma \in \SetPDF{0}{\beta \ell \base 1}{\ell}$. Then the path $(-\delta \ell \base 1 , 0 ) \concat \gamma \concat ( \beta \ell \base 1, (\beta+\delta) \base 1 )$ admits a decomposition of the form
	\begin{equation}
		\label{eqn : UpperTail/Limit/Paths/Decomposition}
		-\delta \ell \base 1 = x_0 \Path{\gamma_{1}} x_1  \Path{\gamma_{2}}  \dots  \Path{\gamma_{R}} x_R = (\beta+\delta) \ell \base 1,
	\end{equation}
	where:
	\begin{enumerate}[(i)]
		\item For all $r \in \intint{ 1}{ R }$, $\gamma_r$ is a cylinder path.
		\item The number of factors in the concatenation satisfies \[ R \le \frac1 \delta +2. \]
	\end{enumerate}
\end{Lemma}

\begin{Lemma}
	\label{lem : UpperTail/Limit/Bound}
	For all $0< \eps < 1/8$, for large enough $\ell >0$, for all $x,y \in \clball{0,2\ell}$ and $\norme{x-y} \le \ell' \le 2\ell $, there exists $\hat x, \hat y \in \clball{0, 3\ell}\cap \Z^d$ such that $\norme{\hat x - \hat y} \ge \norme{x-y}$, and any element of $\SetGDC1xy{\ell'}$ is included in an element of $\SetGDC1{\hat x}{\hat y}{\ell' + 8\eps \ell}$.
\end{Lemma}

We first state and prove Lemma~\ref{lem : UpperTail/Limit/HP} that will be a key ingredient in the proof of Lemma~\ref{lem : UpperTail/Limit/Bound}. 
\begin{Lemma}
	\label{lem : UpperTail/Limit/HP}
	For all $0<\eps < 1$, there exists $0<  \delta < \eps$ such that for all $\beta\in\intervallefo01$ and $\ell>0$, for all $x\in \clball{-\eps \ell \base 1, \delta \ell}$ and $y\in\clball{(\beta + \eps)\ell \base 1, \delta \ell}$,
	\begin{equation}
		\Gap{0}{\beta \ell \base 1} \cap \clball{0,\ell} \subseteq \Gap xy.
	\end{equation}
\end{Lemma}
\begin{proof}
	By bilinearity of the inner product, it is sufficient to treat the case $\ell = 1$ only. Fix $0<\eps < 1$. Let $x\in \clball{-\eps \base 1, \delta }$, $y\in\clball{(\beta + \eps) \base 1, \delta }$ and $z\in \Gap{0}{\beta  \base 1} \cap \clball{0,1}$. Then
	\begin{align*}
		\ps{y-x}{z-x}%
			&= \ps{ (\beta +\eps)\base 1 + \eps \base 1 + \p{y-(\beta +\eps)\base 1} - (\eps \base 1 +x) }{ z+ \eps \base 1 - (\eps \base 1 -x) }\\
			&\ge \ps{(\beta + 2\eps)\base 1}{ z + \eps \base1} - \Cr{HYPERPLANS} \delta,
		\intertext{where $\Cl{HYPERPLANS} > 0$ is a constant. Consequently, if $\delta \le \frac{2\eps^2}{\Cr{HYPERPLANS}}$ then }
		\ps{y-x}{z-x}%
			&\ge 2\eps^2 - \Cr{HYPERPLANS}\delta \ge 0.
	\end{align*}
	Similarly, under the same condition, $\ps{y-x}{z-y}\le 0$, hence $z \in \Gap xy$.
\end{proof}

\begin{proof}[Proof of Lemma~\ref{lem : UpperTail/Limit/Bound}]
	We only treat the case $\cG = \cA^*$, since the other one is nearly identical. Fix $0<\eps < 1$. Let $0<\delta<\eps$ be as in Lemma~\ref{lem : UpperTail/Limit/HP}. Let $\ell \ge \frac{\sqrt d}{2\delta}$. Note that all balls of radius $2\delta \ell$ intersect $\Z^d$. Let $x,y \in \clball{0,2\ell}$ and $\norme{x-y} \le \ell' \le 2\ell$. Let $\xi = (V,E)\in \SetADCalt1{x}{y}{\ell}$. Consider
	\begin{equation}
		(x,y) \in \argmin_{(\hat x, \hat y) \in V^2} \norme{\hat x-\hat y}.
	\end{equation}
	Let $e \dpe \frac{y-x}{\norme{y-x}}$. Let $\hat x \in \clball{ x -2\eps \ell e, 2\delta \ell}\cap \Z^d$ and $\hat y \in \clball{ y +2\eps \ell e, 2\delta \ell}\cap \Z^d$. Up to translation and rotation, Lemma~\ref{lem : UpperTail/Limit/HP} implies that 
	\begin{equation}
		\label{eqn : UpperTail/Limit/Bound/Inclusion_Gap}
		\Gap xy  \cap \clball{0,2\ell} \subseteq \Gap{\hat x}{\hat y}.
	\end{equation}

	We first claim that
	\begin{equation}
		\label{eqn : UpperTail/Limit/Bound/xi_in_SetADCalt}
		\xi \in \SetADCalt1{\hat x}{\hat y}{\ell' + 8\eps \ell}.
	\end{equation}
	Indeed, by triangle inequality,
	\begin{equation*}
		\norme\xi + \d(\hat x, \xi) + \d(\hat y, \xi) \le \ell' + 8\eps \ell, 
	\end{equation*}
	Moreover, by definition of $x$ and $y$, $\xi \subseteq \Gap xy$, thus by the inclusion~\eqref{eqn : UpperTail/Limit/Bound/Inclusion_Gap}, \eqref{eqn : UpperTail/Limit/Bound/xi_in_SetADCalt} holds.

	Inclusion~\eqref{eqn : UpperTail/Limit/Bound/Inclusion_Gap} also implies that
	\begin{equation*}
		\norme{\hat x-\hat y} \ge \norme{x-y}.
	\end{equation*}
	By triangle inequality, since $\xi \subseteq \clball{0,2\ell}$,
	\begin{equation*}
		\hat x,\hat y \in \clball{0, 3\ell},
	\end{equation*}
	which concludes the proof.
\end{proof}

\begin{proof}[Proof of Proposition~\ref{prop : UpperTail/Limit} ]
	\emph{Case 1: Assume that $\mathrm G = \mathrm A^{(q)}$, with $q\in \intervalleff0\infty$.} Fix $\beta \in \intervallefo01$ and $\zeta \in \intervalleoo0\infty$. Let $0<\eps < 1/8$ and $\ell$ be large enough so that the conclusion of Lemma~\ref{lem : UpperTail/Limit/Bound} holds. Let $\xi =(V,E) \in \SetADFalt0{\ell \beta \base 1}\ell$ be a non empty animal, and $x,y \in V \cup \acc{0, \ell \beta \base 1} $ be such that $\norme{x-y} = \diam \p{V \acc{0, \ell \beta \base 1} } $. Then $\norme{x-y} \ge \beta \ell$, and
	\[ \xi \in \SetADCalt1xy\ell. \]
	Consequently, by Lemma~\ref{lem : UpperTail/Limit/Bound} and union bound,
	\begin{align}
		\Pb{ \MassADFpen0{\ell \beta e}\ell q \ge \zeta \ell }%
			&\le \sum_{ \begin{array}{c} \hat x,\hat y \in \clball{0,3\ell}\cap  \Z^d \\ \norme{ \hat x- \hat y} \ge \beta \ell	\end{array}  }%
				\Pb{ \MassADCpen1{\hat x}{\hat y}{ (1+8\eps)\ell }q \ge \zeta \ell } \eol
			&\le \Pol(\ell) \max_{ \begin{array}{c} \hat x,\hat y \in \clball{0,3\ell}\cap  \Z^d \\ \norme{ \hat x- \hat y} \ge \beta \ell	\end{array}  }
				\Pb{ \MassADCpen1{\hat x}{\hat y}{ (1+8\eps)\ell }q \ge \zeta \ell }, \nonumber
	\end{align}
	where $\Pol$ is a polynomial that only depends on $d$. Consequently,
	\begin{align}
		&-\frac 1\ell \log \Pb{ \MassADFpen0{\ell \beta \base1 }\ell q \ge \zeta \ell} \eol
			&\quad \ge - \frac{\log \Pol(\ell)}{\ell}  + \min_{ \begin{array}{c} \hat x,\hat y \in \clball{0,3\ell}\cap  \Z^d \\ \norme{ \hat x- \hat y} \ge \beta \ell	\end{array}  }  -\frac1\ell \log \Pb{ \MassADCpen1{\hat x}{\hat y}{ (1+8\eps)\ell }q \ge \zeta \ell }. \nonumber
	\end{align}
	Plugging in~\eqref{eqn : UpperTail/LimitDiamAntidiam/Diam}, we get
	\begin{align}
		-\frac 1\ell \log \Pb{ \MassADFpen0{\ell \beta \base 1}\ell q \ge \zeta \ell}%
			&\ge - \frac{\log \Pol(\ell)}{\ell}  + (1+8\eps) \FdT_{ \rA^{(q)} }\p{ \frac{\beta }{1 +8\eps}, \frac{\zeta }{1+8\eps} }.\nonumber
		\intertext{Letting $\ell \to \infty$ gives}
		\liminf_{\ell \to \infty} -\frac1\ell \log \Pb{ \MassADFpen0{\ell \beta \base 1}\ell q \ge \zeta \ell }%
			&\ge (1+8\eps) \FdT_{ \rA^{(q)} }\p{ \frac{\beta }{1 +8\eps}, \frac{\zeta }{1+8\eps} }.
		\intertext{The continuity of $\FdT$ then provides the inequality}
		\liminf_{\ell \to \infty }-\frac 1\ell \log \Pb{ \MassADFpen0{\ell \beta \base 1}\ell q \ge \zeta \ell}%
			&\ge \FdT_{ \rA^{(q)} }\p{ \beta, \zeta }.
		\intertext{The converse inequality}
		\limsup_{\ell \to \infty }-\frac 1\ell \log \Pb{ \MassADFpen0{\ell \beta \base1}\ell q \ge \zeta \ell}%
			&\le \FdT_{ \rA^{(q)} }\p{ \beta, \zeta }
	\end{align}
	is straightforward and completes the proof of~\eqref{eqn : UpperTail/Limit}.

	\emph{Case 2: Assume that $\rG = \rP$.} Fix $\beta \in \intervallefo01$. Let $0 < \eps , \delta < 1/8$ and $\ell$ be large enough so that the conclusion of Lemma~\ref{lem : UpperTail/Limit/Bound} holds and $\eps\ell \ge \frac1\delta +2$ . An $R$-uple
	\begin{equation*}
		( y_r, z_r, \ell_r)_{1\le r \le R} \in \p{ \p{ \clball{0, 3\ell} \cap \Z^d}^2 \times \intint{1}{ \ceil{3\ell} } }^R
	\end{equation*}
	is said to be an \emph{eligible skeleton} if the three following assertions hold:
	\begin{align}
		\label{eqn : UpperTail/Limit/Eligible1}
			R &\le R_0 \dpe \frac1\delta +2,\\
		\label{eqn : UpperTail/Limit/Eligible3}
			\sum_{r=1}^R \norme{y_r - z_r} &\ge \beta \ell, \\
		\label{eqn : UpperTail/Limit/Eligible4}
			\sum_{r=1}^R \ell_r &\le \ell + 9R\eps \ell + 2\delta \ell.
	\end{align}
	We denote by \[ \cE = \cE(\ell, \eps, \delta, \beta) = \acc{ \p{ y_{r}^E, z_{r}^E,  \ell_{r}^E }_{1\le r \le R_E} }_{ E \in \cE  } \] the set of eligible skeletons. We first claim that $\SetPDF{0}{\ell \beta \base 1}{\ell}$ is decomposable in the family
	\begin{equation}
		\p{ \SetPDC1{y_r^E}{z_r^E}{\ell_r^E} }_{ \substack{E \in \cE \\ 1\le r \le R_E} },
	\end{equation}
	in the sense of Definition~\ref{def : Concatenation/Decomposition}. Indeed,	let $\gamma \in \SetPDF0{\ell \beta \base 1}\ell$. Consider the decomposition~\eqref{eqn : UpperTail/Limit/Paths/Decomposition} and denote, for all $r \in \intint1R$, $\ell_r \dpe \norme{\gamma_r}$. We have
	\begin{equation}
		\sum_{r=1}^R \ell_r \le (1+2\delta)\ell.
	\end{equation}
	By Lemma~\ref{lem : UpperTail/Limit/Bound}, for all $r\in \intint1R$ there exists $y_r, z_r \in \clball{0, 3\ell}\cap \Z^d$ such that $\norme{y_r - z_r} \ge \norme{x_{r-1} - x_r}$, and $\gamma_r$ has an extension in $\SetPDC1{y_r}{z_r}{\ell_r + 8\eps \ell}$. Denote $\hat \ell_r \dpe \ceil{\ell_r + 8\eps \ell}$. It is straightforward to check that $\p{y_r, z_r, \hat \ell_r}_{1\le r \le R}$ is an eligible skeleton, thus the claim holds.

	Let $\zeta>0$. By Lemma~\ref{lem : Decomposition} we have
	\begin{equation}
		\label{eqn : UpperTail/Limit/Paths/BK}
		\Pb{\MassPDF0{\ell \beta \base 1}\ell \ge \zeta \ell} \le \sum_{E \in \cE} \sum_{\substack{ (t_r) \in \N^{R_E} \\ \zeta \ell -R_E \le \sum t_r \le \zeta \ell} } \prod_{r=1}^{R_E} \Pb{ \MassPDC1{y_{r}^E} {z_{r}^E}{ \ell_{r}^E } \ge t_r }.
	\end{equation}
	Taking the $\log$ and multiplying by $-\frac1\ell$, we get
	\begin{align}
		\nonumber
		\begin{split}
		&-\frac1\ell \log \Pb{\MassPDF0{\ell \beta \base 1}\ell \ge \zeta \ell} \\
			&\quad \ge -\frac{\Pol(\ell)}{\ell} \min_{E \in \cE} \min_{\substack{ (t_r) \in \N^{R_E} \\ \zeta \ell -R_E \le \sum t_r \le \zeta \ell} } - \sum_{r=1}^{R_E} \frac1\ell \log \Pb{ \MassPDC1{y_{r}^E} {z_{r}^E}{ \ell_{r}^E } \ge t_r },
		\end{split}
		\intertext{where $\Pol(\cdot)$ is a polynomial that only depends on $d$ and $\delta$. Now plugging in~\eqref{eqn : UpperTail/LimitDiamAntidiam/Diam} gives}
		\label{eqn : UpperTail/Limit/Paths/Inegalite_Moche}
		\begin{split}
		&-\frac1\ell \log \Pb{\MassPDF0{\ell \beta \base 1}\ell \ge \zeta \ell} \\
			&\quad \ge -\frac{\log \Pol(\ell)}{\ell}  + \min_{E \in \cE} \min_{\substack{ (t_r) \in \N^{R_E} \\ \zeta \ell -R_E \le \sum t_r \le \zeta \ell} } \sum_{r=1}^{R_E} \frac{ \ell_{r}^E }{\ell} \FdT_\rP\p{ \frac{\norme{y_{r}^E - z_{r}^E } }{ \ell_{r}^E } , \frac{ t_r }{ \ell_{r}^E } }.
		\end{split}
	\end{align}
	We aim at lower bounding the sum in the right-hand side. To so, let us fix $E\in \cE$ and $(t_r) \in \N^{R_E}$ such that $\zeta \ell - R_E \le \sum t_r \le \zeta \ell$. To simplify the notations, we drop the superscript $E$. For all $r\in \intint1R$, define
	\begin{equation}
		\theta_r \dpe \frac{ \ell_r}{ \sum_{s=1}^R  \ell_s}.
	\end{equation}
	By convexity of $\FdT$,
	\begin{align}
		\sum_{r=1}^{R} \frac{ \ell_{r} }{\ell} \FdT_\rP\p{ \frac{\norme{y_{r} - z_{r} } }{ \ell_{r} } , \frac{ t_r }{  \ell_{r} } }%
			&=  \frac{ \sum_{s=1}^{R}  \ell_{s} }{\ell} \cdot \sum_{r=1}^R \theta_r \FdT_\rP\p{ \frac{\norme{y_{r} - z_{r} } }{ \ell_{r} } , \frac{ t_r }{  \ell_{r} } }\eol
			&\ge \frac{ \sum_{s=1}^{R}  \ell_{s} }{\ell} \cdot \FdT_\rP\p{ \sum_{r=1}^R\frac{\theta_r  \norme{y_{r} - z_{r} }}{  \ell_{r} } ,  \sum_{r=1}^R \frac{ \theta_r t_r }{  \ell_{r} } } \eol
			&= \frac{ \sum_{s=1}^{R}  \ell_{s} }{\ell} \cdot \FdT_\rP\p{ \frac{ \sum_{r=1}^R \norme{y_{r} - z_{r} }}{ \sum_{s=1}^R  \ell_{s} } ,  \frac{ \sum_{r=1}^R t_r }{ \sum_{s=1}^R  \ell_{s} } }.\nonumber
		\intertext{Using~\eqref{eqn : UpperTail/Limit/Eligible1},~\eqref{eqn : UpperTail/Limit/Eligible3} and the lower bound on $\sum t_r$, since $\FdT$ is nondecreasing on $\intervallefo01 \times \intervalleoo0\infty$, }
		\sum_{r=1}^{R} \frac{ \ell_{r} }{\ell} \FdT_\rP\p{ \frac{\norme{y_{r} - z_{r} } }{ \ell_{r} } , \frac{ t_r }{  \ell_{r} } }%
			&\ge \frac{ \sum_{s=1}^{R}  \ell_{s} }{\ell} \cdot \FdT_\rP\p{ \frac{ \beta \ell }{ \sum_{s=1}^R  \ell_{s} } ,  \frac{ (\zeta - \eps)\ell }{ \sum_{s=1}^R  \ell_{s} } }.\nonumber
		\intertext{Using~\eqref{eqn : UpperTail/Limit/Eligible4} and the convexity of $\FdT$ again leads to }
		\sum_{r=1}^{R} \frac{ \ell_{r} }{\ell} \FdT_\rP\p{ \frac{\norme{y_{r} - z_{r} } }{ \ell_{r} } , \frac{ t_r }{  \ell_{r} } }
			&\ge (1+9R_0 \eps +2\delta)\FdT_\rP\p{ \frac{ \beta }{ 1+9R_0 \eps +2\delta } ,  \frac{ \zeta - \eps }{ 1+9R_0 \eps +2\delta } }.
		\label{eqn : UpperTail/Limit/Paths/Inegalite_Moche/LB_terme_general}
	\end{align}
	Note that the right-hand side in \eqref{eqn : UpperTail/Limit/Paths/Inegalite_Moche/LB_terme_general} does not depend on $E \in \cE$, thus plugging it into~\eqref{eqn : UpperTail/Limit/Paths/Inegalite_Moche} and letting $\ell \to \infty$ provides
	\begin{align}
		\liminf_{\ell \to \infty} -\frac1\ell \log \Pb{\MassPDF0{\ell \beta \base 1}\ell \ge \zeta \ell} &\ge (1+9R_0 \eps+ 2 \delta)\FdT_\rP\p{ \frac{ \beta }{ 1+9R_0 \eps + 2 \delta} ,  \frac{ \zeta - \eps }{ 1+9R_0 \eps + 2 \delta} }.\nonumber
		\intertext{Since $\FdT$ is continuous, setting $\delta \dpe \sqrt \eps$ and letting $\eps \to 0$ gives }
		\liminf_{\ell \to \infty} -\frac1\ell \log \Pb{\MassPDF0{\ell \beta \base 1}\ell \ge \zeta \ell} &\ge \FdT_\rP\p{ \beta ,  \zeta }.
		\intertext{The converse inequality }
		\limsup_{\ell \to \infty} -\frac1\ell \log \Pb{\MassPDF0{\ell \beta \base 1}\ell \ge \zeta \ell} &\le \FdT_\rP\p{ \beta ,  \zeta }
	\end{align}
	being straightforward, this concludes the proof of~\eqref{eqn : UpperTail/LimitFree}.
\end{proof}


\subsection{Order of the upper-tail large deviations}
\label{subsec : UpperTail/OdG}

In this section we prove Proposition~\ref{prop : UpperTail/OdG}. In the cases where we have access to concentration inequalities, i.e.\ when $\nu$ has a bounded support, and $\rG = \rP$ or $\rA^{(q)}$ for some $q \in \intervallefo0\infty$, it is a clear consequence of Theorem~\ref{thm : BOUNDS}. We get to the general cases by standard truncation arguments. The unbounded term will be handled thanks to Lemma~\ref{lem : UpperTail/OdG/Unbounded}. Given $m>0$, we define $\nu_1^{(m)}$ and $\nu_2^{(m)}$ as the respective images of $\nu$ by
\begin{alignat*}{2}
	f_1 : \intervalleoo0\infty &\longrightarrow \intervalleof0m &\quad \text{  and  }\quad  f_2 : \intervalleoo0\infty &\longrightarrow \intervallefo0\infty \\
	 t &\longmapsto t\wedge m & t &\longmapsto (t- m)^+.
\end{alignat*}
Similarly we define $\pro_1^{(m)}$ and $\pro_2^{(m)}$ as the respective images of $\pro$ by
\begin{alignat*}{2}
	\overline f_1 : \R^d \times \intervalleoo0\infty &\longrightarrow \R^d \times \intervalleof0m &\quad \text{  and  }\quad  \overline f_2 : \R^d \times \intervalleoo0\infty &\longrightarrow \R^d \times \intervallefo0\infty \\
	 (x,t) &\longmapsto (x,t\wedge m) & (x,t) &\longmapsto (x,(t- m)^+).
\end{alignat*}

By the so-called mapping theorem (see e.g.\ Theorem 5.1 in \cite{LastPenrose}), $\pro_1^{(m)}$ and $\pro_2^{(m)}$ are Poisson point processes with respective intensities $\Leb\otimes \nu_1^{(m)}$ and $\Leb\otimes \nu_2^{(m)}$.
\begin{Lemma}
	\label{lem : UpperTail/OdG/Unbounded}
	For all $\eps>0$, there exists $m>0$ such that
	\begin{equation}
		\label{eqn : UpperTail/OdG/Unbounded}
		\liminf_{\ell \to \infty} -\frac1\ell \log\Pb{ \MassAUF \ell\cro{\pro_2^{(m)}} \ge \eps \ell} >0.
	\end{equation}
\end{Lemma}
\begin{proof}
	Let $\lambda >0$ be such that~\eqref{ass : ExpMoment} holds and $m>0$. For all $v\in \Z^d$, we define \[ X_v \dpe \Mass{v +\intervalleff01^d}.\] There exists a constant $\Cl{ContinuousVSLattice}>0$ such that for large enough $\ell$, for all $\xi \in \SetAUF \ell$, there exists a \emph{lattice} animal $\xi'$ such that 
	\begin{equation}
		\label{eqn : UpperTail/OdG/Unbounded/ContVSLatt}
		0\in \xi', \quad \#\xi' \le \Cr{ContinuousVSLattice}\ell \text{ and } \xi \subseteq \bigcup_{v\in \xi'} \p{v+ \intervallefo01^d}. 
	\end{equation}
	Furthermore, by Proposition 2.1.4 in \cite{Bac20},
	\begin{equation}
		\E{ \exp\p{\lambda X_0 } } =  \exp\p{ \int_{\intervallefo0\infty}\p{\mathrm e^{\lambda t} -1 }\nu(\d t) } < \infty.
	\end{equation}
	The family $(X_v)_{v\in \Z^d}$ is i.i.d, therefore by Lemma~4.2 in \cite{Dem01}, there exists $m>0$ such that
	\[ \liminf_{\ell \to \infty}-\frac1\ell \log \Pb{  \begin{array}{c} \text{ There exists a lattice animal $\xi'$ containing $0$, with cardinal}\\ \text{at most $\Cr{ContinuousVSLattice}\ell$, such that}  \sum_{v \in \xi'} \p{ X_v - m}^+ \ge \eps \ell  \end{array}}>0,  \]
	thus by the straightforward bound
	\[
	X_v\cro{\pro_2^{(m)} } \le \p{X_v - m}^+,
	\]
	there exists $m>0$ such that
	\begin{equation}
		\label{eqn : UpperTail/OdG/Unbounded/DGK}
		\liminf_{\ell \to \infty}-\frac1\ell \log \Pb{  \begin{array}{c} \text{ There exists a lattice animal $\xi'$ containing $0$, with cardinal}\\ \text{at most $\Cr{ContinuousVSLattice}\ell$, such that }  \sum_{v \in \xi'} X_v\cro{\pro_2^{(m)}} \ge \eps \ell \end{array} }>0.
	\end{equation}

	Combining~\eqref{eqn : UpperTail/OdG/Unbounded/ContVSLatt} and~\eqref{eqn : UpperTail/OdG/Unbounded/DGK} concludes the proof.
\end{proof}

\begin{proof}[Proof of Proposition~\ref{prop : UpperTail/OdG}]

Let $\zeta> \LimMassG(\beta)$ and $0 < \eps < \zeta - \LimMassG(\beta)$.

\emph{Case 1: $\rG = \rP$ or $\rA^{(q)}$, for some $q \in \intervallefo0\infty$.}
Let $m>0$ be such that~\eqref{eqn : UpperTail/OdG/Unbounded} holds. Let $\ell>0$. By union bound,
\begin{equation*}
	\Pb{\MassGDF0{\ell \beta \base 1}\ell \ge \zeta \ell } %
		\le \Pb{ \MassGDF0{\ell \beta \base 1}\ell\cro{\pro_1^{(m)}} \ge (\zeta-\eps) \ell } + \Pb{ \MassAUF\ell\cro{\pro_2^{(m)}} \ge \eps \ell }
\end{equation*}
Consequently,
\begin{equation*}
	\FdT(\beta, \zeta) \ge \min\p{ \FdT(\beta, \zeta - \eps)\cro{\nu_1^{(m)}}, \liminf_{\ell \to \infty}-\frac1\ell \log \Pb{ \MassAUF\ell\cro{\pro_2^{(m)}} \ge \eps \ell } }.
\end{equation*}
The first term is positive by Theorem~\ref{thm : BOUNDS} and the second is positive by definition of $m$, thus $\FdT(\beta, \zeta)>0$.

\emph{Case 2: $\rG = \rA^{(\infty).}$}
By Theorem 1.7 in \cite{Article3} there exists $q\in \intervallefo0\infty$ such that \[ \zeta > \LimMassA^{(q)}(\beta). \]
Furthermore, for all $\ell >0$,
\begin{equation}
	\acc{ \MassADFpen0{\ell \beta \base 1}\ell \infty\ge \zeta \ell} \subseteq \acc{ \MassADFpen0{\ell \beta \base 1}\ell q \ge \zeta \ell},
\end{equation}
thus we can conclude with Case 1.
\end{proof}


\appendix
\section{Large deviation principle at speed $n$}
\label{appsec : LDP}
In this section we assume that~\eqref{ass : ExpMoment} holds, fix $\beta \in \intervallefo01$ and prove Corollary~\ref{cor : LDP}. Our strategy consists in proving that $\frac1\ell \MassGDF{0}{\ell \beta \base 1}{\ell}$ follows the so-called \emph{weak LDP}, then checking that is it \emph{exponentially tight} to show the full LDP.

\begin{Lemma}
	\label{lem : LDP/weakLDP}
	The process $\p{ \frac1\ell \MassGDF{0}{\ell \beta \base 1}{\ell} }_{\ell >0}$ follows the weak LDP at speed $n$ with the rate function $\FdT^*(\beta, \cdot)$, i.e.\
	\begin{enumerate}[(i)]
		\item For all compact subsets $K\subseteq \intervallefo0\infty$, \[ \liminf_{\ell \to \infty}-\frac1\ell \log\Pb{ \frac1\ell \MassGDF{0}{\ell \beta \base 1}{\ell} \in K} \ge \inf_{\zeta \in K} \FdT^*(\beta, \zeta). \]
		\item For all open subsets $U\subseteq \intervallefo0\infty$, \[ \liminf_{\ell \to \infty}-\frac1\ell \log\Pb{ \frac1\ell \MassGDF{0}{\ell \beta \base 1}{\ell} \in U} \le \inf_{\zeta \in U} \FdT^*(\beta, \zeta). \]
	\end{enumerate}
\end{Lemma}
\begin{proof}
	For all $\zeta \in \intervallefo0\infty$, we define
	\begin{align}
		\overline I(\zeta) &\dpe \lim_{\eps \to 0} \limsup_{\ell \to \infty} -\frac1\ell \log \Pb{ \frac1\ell \MassGDF{0}{\ell \beta \base 1}{\ell} \in \intervalleff{\zeta - \eps}{\zeta + \eps} }
		\intertext{and}
		\underline I(\zeta) &\dpe \lim_{\eps \to 0} \liminf_{\ell \to \infty} -\frac1\ell \log \Pb{ \frac1\ell \MassGDF{0}{\ell \beta \base 1}{\ell} \in \intervalleff{\zeta - \eps}{\zeta + \eps} }.
	\end{align}
	Note that both limits exist by monotonicity. It is sufficient by Theorem 4.1.11 in \cite{LDTA} to check that for all $\zeta \in \intervallefo0\infty$, 
	\begin{equation} \label{eqn : LDP/weakLDP/Equality} \overline I = \underline I = \FdT^*(\beta, \cdot).  \end{equation}
	Let $\zeta \in \intervallefo0\infty$. We distinguish three cases, according to the relative position of $\zeta$ and $\LimMassG(\beta)$.

	\emph{Case 1: $\zeta < \LimMassG(\beta)$.} Let $\eps>0$ be small enough so that $\zeta + \eps < \LimMassG(\beta)$. Then by the inclusion
	\[
		\acc{ \frac1\ell \MassGDF{0}{\ell \beta \base 1}{\ell} \in \intervalleff{\zeta - \eps}{\zeta + \eps} } \subseteq \acc{\MassGDF{0}{\ell \beta \base 1}\ell \le (\zeta+\eps) \ell }
	\]
	and Theorem~\ref{thm : MAIN/LowerTail}, $\overline I (\zeta) = \underline I(\zeta) = \infty$, thus~\eqref{eqn : LDP/weakLDP/Equality}.

	\emph{Case 2: $\zeta = \LimMassG(\beta)$.} By~\eqref{eqn : Intro/LLN/Directed}, for all $\eps > 0$, \[ \lim_{\ell \to \infty} \Pb{ \frac1\ell \MassGDF{0}{\ell \beta \base 1}{\ell} \in \intervalleff{\zeta - \eps}{\zeta + \eps} } = 1. \] Consequently, $\overline I (\zeta) = \underline I(\zeta) = 0$, thus~\eqref{eqn : LDP/weakLDP/Equality}.

	\emph{Case 3: $\zeta > \LimMassG(\beta)$.} Let $\eps>0$ be small enough so that $\zeta - \eps > \LimMassG(\beta)$. We have
	\[
		\acc{ \frac1\ell \MassGDF{0}{\ell \beta \base 1}{\ell} \in \intervalleff{\zeta - \eps}{\zeta + \eps} } = \acc{\MassGDF{0}{\ell \beta \base 1}\ell \ge (\zeta-\eps) \ell } \setminus \acc{\MassGDF{0}{\ell \beta \base 1}\ell > (\zeta+\eps) \ell }.
	\]
	Moreover, by strict monotonicity of $\FdT(\beta, \cdot)$ on $\intervallefo{\LimMassG(\beta)}{\infty}$ (see Theorem~\ref{thm : MAIN/UpperTail}\eqref{item : MAIN/UpperTail/Limit}), $\FdT(\beta, \zeta-\eps) < \FdT(\beta, \zeta+\eps)$, therefore by~\eqref{eqn : MAIN/UpperTail/Directed},
	\[ \Pb{\MassGDF{0}{\ell \beta \base 1}\ell \ge (\zeta-\eps) \ell } \gg \Pb{\MassGDF{0}{\ell \beta \base 1}\ell \ge (\zeta+\eps) \ell }  \text{ as } \ell \to \infty.\]
	Hence we have
	\begin{equation}
		\label{eqn : LDP/weakLDP/LB}
		\lim_{\ell \to \infty} -\frac1\ell \log \Pb{ \frac1\ell \MassGDF{0}{\ell \beta \base 1}{\ell} \in \intervalleff{\zeta - \eps}{\zeta + \eps} } = \FdT(\beta, \zeta -\eps).
	\end{equation}
	Letting $\eps \to 0$ and using the continuity of $\FdT$, we get $\overline I (\zeta) = \underline I(\zeta) = \FdT(\beta, \zeta)$, thus~\eqref{eqn : LDP/weakLDP/Equality}.
\end{proof}
\begin{Lemma}
	\label{lem : LDP/Tight}
	The process $\p{ \frac1\ell \MassGDF{0}{\ell \beta \base 1}{\ell} }_{\ell >0}$ is exponentially tight, i.e.\ for all $\alpha >0$, there exists a compact subset $K\subseteq \intervallefo0\infty$ such that
	\[ \liminf_{\ell \to \infty}-\frac1\ell \log\Pb{ \frac1\ell \MassGDF{0}{\ell \beta \base 1}{\ell} \in K^\mathrm{c}} \ge \alpha. \]
\end{Lemma}
\begin{proof}
	Let $\alpha>0$. By Theorem \ref{thm : MAIN/UpperTail}\eqref{item : MAIN/UpperTail/Limit}, there exists $\zeta \in \intervalleoo0\infty$ such that $\FdT(\beta, \zeta) \ge \alpha$. Let $K \dpe \intervalleff0\zeta$. By~\eqref{eqn : MAIN/UpperTail/Directed} we have
	\[ \liminf_{\ell \to \infty}-\frac1\ell \log\Pb{ \frac1\ell \MassGDF{0}{\ell \beta \base 1}{\ell} \in K^\mathrm{c}} \ge \FdT(\beta, \zeta) \ge  \alpha.\]
\end{proof}
\begin{proof}[Proof of Corollary~\ref{cor : LDP}]
	By Lemma 1.2.18 in \cite{LDTA} and the remark below, a exponentially tight process following the weak LDP with a given rate function actually follows the LDP with this rate function. Lemma~\ref{lem : LDP/weakLDP} and \ref{lem : LDP/Tight} state that it is the case for $\frac1\ell \MassGDF{0}{\ell \beta \base 1}{ \ell}$, with the rate function $\FdT^*(\beta, \cdot)$.
\end{proof}

\newpage
\bibliographystyle{plainurl}
\bibliography{biblio}
\end{document}